\documentclass[12pt]{amsart}
\usepackage{graphicx,amssymb,latexsym,bm,color}
\usepackage[dvips,centering,includehead,width=15.1cm,height=22.7cm]{geometry}
\usepackage{amsmath}
\usepackage{graphicx}
\usepackage{epsfig}
\usepackage{amssymb}
\input{xy} \xyoption{all}

\numberwithin{equation}{section}

\newtheorem{theorem}{Theorem}[section]
\newtheorem{lemma}[theorem]{Lemma}

\newtheorem{proposition}[theorem]{Proposition}

\theoremstyle{definition}
\newtheorem{definition}[theorem]{Definition}
\newtheorem{example}[theorem]{Example}
\newtheorem{remark}[theorem]{Remark}

\newcommand{\eps}{{\varepsilon}}

\renewcommand{\L}{{\mathcal L}}

\newcommand{\D}{{\mathcal D}}
\renewcommand{\P}{{\mathcal P}}

  \newcommand{\defect}{{\mathrm{def}}}
  
 \newcommand{\wls}{{\text{wls}}}
 \newcommand{\nr}{{\text{nr}}}
 \newcommand{\wt}{{\text{wt}}}
 \newcommand{\len}{{\text{len}}}

\newcommand{\CC}{{\mathbb C}}
\newcommand{\RR}{{\mathbb R}}
\newcommand{\ZZ}{{\mathbb Z}}

\newcommand{\PP}{{\mathbb P}}

\renewcommand{\o}{\multiput(0,0)(10,0){1}{\circle{2}}}
\newcommand{\oo}{\multiput(0,0)(10,0){2}{\circle{2}}}
\newcommand{\ooo}{\multiput(0,0)(10,0){3}{\circle{2}}}
\newcommand{\oooo}{\multiput(0,0)(10,0){4}{\circle{2}}}

\newcommand{\Eeee}{\put(1,0){\line(1,0){8}}}
\newcommand{\eEee}{\put(11,0){\line(1,0){8}}}

\newcommand{\eOee}{\qbezier(10.8,0.6)(15,4)(19.2,0.6)
\qbezier(10.8,-0.6)(15,-4)(19.2,-0.6)}

\newcommand{\eeOe}{\qbezier(20.8,0.6)(25,4)(29.2,0.6)
\qbezier(20.8,-0.6)(25,-4)(29.2,-0.6)}

  \title{Relative Node Polynomials for Plane Curves} \author{Florian Block}
\date{\today}
\address{Mathematics Institute, University of Warwick, Coventry, CV4
  7AL, United Kingdom}
  \email{f.s.block@warwick.ac.uk}
\thanks{The author was partially supported by a Rackham One-Term
  Dissertation Fellowship \&  by the NSF grant DMS-055588.}

\begin{document}

\begin{abstract}
We generalize the recent work of S.~Fomin and G.~Mikhalkin on polynomial formulas for Severi degrees.

The degree of the Severi variety of plane curves of degree $d$ and $\delta$ nodes is given by a polynomial in $d$, provided $\delta$ is fixed and $d$ is large enough. We extend this result to generalized Severi varieties parametrizing plane curves which, in addition, satisfy tangency conditions of given orders with respect to a given line. We show that the degrees of these varieties, appropriately rescaled, are given by a combinatorially defined ``relative node polynomial'' in the tangency orders, provided the latter are large enough. We describe a method to compute these polynomials for arbitrary $\delta$, and use it to present explicit formulas for $\delta \le 6$. We also give a threshold for polynomiality, and compute the first few leading terms for any~$\delta$.
\end{abstract}

\maketitle

\section{Introduction and Main Results}
\label{sec:intro}

The \emph{Severi degree} $N^{d, \delta}$ is the degree of
the Severi variety of (possibly reducible)
nodal plane curves of degree $d$ with $\delta$ nodes. Equivalently, $N^{d, \delta}$ is the number of such curves 
passing through
$\tfrac{(d+3)d}{2} - \delta$ generic points in the complex projective
plane~$\CC \PP^2$. Severi varieties have 
received considerable attention since they were introduced by F.~Enriques~\cite{En12} and F.~Severi~\cite{Se21} around 1915. Much later, in 1986, J.~Harris~\cite{Ha86} achieved a celebrated breakthrough by 
showing their irreducibility.

In 1994, P.~Di~Francesco and C.~Itzykson \cite{DI} conjectured that
the numbers $N^{d, \delta}$ are given by a polynomial in $d$, for a fixed number of nodes
$\delta$, provided $d$ is large enough.
S.~Fomin and
G.~Mikhalkin \cite[Theorem 5.1]{FM} established this
polynomiality in 2009. More precisely, they showed that there exists, for
every $\delta \ge 1$, a \emph{node polynomial} $N_\delta(d)$ which
satisfies $N^{d, \delta} = N_\delta(d)$, for all $d \ge 2 \delta$.

The polynomiality of $N^{d, \delta}$ and the polynomials $N_\delta(d)$
were known in the 19th century for $\delta = 1, 2$ and $3$. For $\delta =
4, 5$ and $6$, this was only achieved by
I.~Vainsencher~\cite{Va} in 1995. In 2001, S.~Kleiman and R.~Piene
\cite{KP} settled the cases $\delta = 7$ and $8$. In \cite{FB}, the author  computed $N_\delta(d)$
for $\delta \le 14$ and improved the threshold of S.~Fomin and G.~Mikhalkin
by showing that $N^{d, \delta} = N_\delta(d)$ provided $d \ge \delta$.

Severi degrees can be generalized to incorporate tangency
conditions to a fixed line $L \subset
\CC \PP^2$. More specifically, the \emph{relative Severi degree} $N^{\delta}_{\alpha, \beta}$ is the number of (possibly
reducible) nodal plane curves with $\delta$ nodes that have
tangency of order~$i$ to $L$ at $\alpha_i$  \emph{fixed} points (chosen in advance) and
tangency of order $i$ to $L$ at $\beta_i$ \emph{unconstrained} points, for all $i \ge 1$, and that pass through an
appropriate number of generic points. Equivalently, $N^{\delta}_{\alpha, \beta}$ is the degree of the
\emph{generalized Severi variety} studied in~\cite{CH98,Va00}. By
B\'ezout's Theorem, the degree of a curve with tangencies of order $(\alpha, \beta)$ equals $d = \sum_{i \ge 1} i(\alpha_i + \beta_i)$. The number of
point conditions (for a potentially finite count) is
$\frac{(d+3)d}{2}- \delta - \alpha_1 - \alpha_2 - \cdots$. 
We recover non-relative Severi degrees by specializing to $\alpha =
(0, 0, \dots)$ and $\beta = (d, 0, 0, \dots)$. The numbers $N^{\delta}_{\alpha, \beta}$ are determined
by the rather complicated Caporaso--Harris recursion~\cite{CH98}.

In this paper, we show that much of the story of (non-relative) node
polynomials carries over to relative Severi
degrees. Our main result is that, up to a simple combinatorial
factor and for fixed $\delta \ge 1$, the relative
Severi degrees $N^\delta_{\alpha, \beta}$ are given by a multivariate \emph{polynomial} in
$\alpha_1, \alpha_2, \dots, \beta_1, \beta_2, \dots$, provided that $\beta_1+ \beta_2+ \dots$ is
sufficiently large. 
For a sequence $\alpha = (\alpha_1, \alpha_2,
\dots)$ of non-negative integers with only finitely many $\alpha_i$
non-zero, we write
\begin{displaymath}
|\alpha|  \stackrel{\text{def}}{=} \alpha_1 + \alpha_2 + \cdots,
\quad \quad \alpha ! \stackrel{\text{def}}{=} \alpha_1 ! \cdot \alpha_2 ! \cdot \cdots.
\end{displaymath}
Throughout the paper, we use the grading
$\deg(\alpha_i) = \deg(\beta_i) = 1$ (so that $d$ and $|\beta|$ are homogeneous of
degree $1$). The following is our main result.

\begin{theorem}
\label{thm:relativenodepoly}
For every $\delta \ge 1$, there is a combinatorially defined
polynomial $N_{\delta}(\alpha; \beta)$ in $\alpha_1,
\alpha_2, \dots, \beta_1, \beta_2, \dots$  of (total) degree
$3\delta $ such that, for all $\alpha_1,
\alpha_2, \dots, \beta_1, \beta_2, \dots$ with $|\beta| \ge \delta $, the relative Severi
degree $N^\delta_{\alpha, \beta}$ is given by
\begin{equation}
\label{eqn:relativenodepoly}
N^\delta_{\alpha, \beta} =1^{\beta_1} 2^{\beta_2} \cdots \frac{(|\beta| - \delta)!}{\beta!}
\cdot N_{\delta}(\alpha_1,
\alpha_2, \dots; \beta_1, \beta_2, \dots).
\end{equation}
\end{theorem}

We call $N_{\delta}(\alpha; \beta)$ the \emph{relative node
  polynomial} and use the same notation as in the non-relative case if
no confusion can occur. 
We do not need to specify the number
of variables in light of the following stability result.

\begin{theorem}
\label{thm:stability}
For $\delta \ge 1$ and vectors $\alpha = (\alpha_1, \dots, \alpha_{m})$,
$\beta = (\beta_1, \dots, \beta_{m'})$ with $|\beta| \ge \delta $, the
following polynomial identities hold:
\begin{displaymath}
N_\delta(\alpha, 0; \beta) = N_\delta(\alpha; \beta) \quad \text{ and
} \quad
N_\delta(\alpha; \beta, 0) = N_\delta(\alpha; \beta).
\end{displaymath}
Therefore, there exists a formal power series
$N^{\infty}_\delta(\alpha; \beta)$ in infinitely many variables
$\alpha_1, \alpha_2, \dots,$ $\beta_1, \beta_2, \dots$ that
specializes to all relative node polynomials under $\alpha_{m +1} =
\alpha_{m+2} = \cdots = 0$ and $\beta_{m' +1} =
\beta_{m'+2} = \cdots = 0$, for various $m, m' \ge 1$.
\end{theorem}

In fact, even more is true.

\begin{proposition}
\label{prop:variables}
For $\delta \ge 1$, the relative node polynomial $N_\delta(\alpha,
\beta)$ is a polynomial in $d$, $|\beta|$, $\alpha_1, \dots,
\alpha_\delta$, and $\beta_1, \dots, \beta_\delta$, where $d = \sum_{i
  \ge 1} i (\alpha_i + \beta_i)$.
\end{proposition}

Using the combinatorial description, we provide a method to
compute the relative node polynomials for any $\delta$ (see
Sections~\ref{sec:relativenodepolys} and~\ref{sec:proofs}). We use it to compute $N_\delta(\alpha;
\beta)$ for $\delta \le 6$. Due to spacial constrains, we only
tabulate the cases $\delta \le 3$. The polynomials $N_0$
and $N_1$ already appeared (implicitly) in \cite[Section~4.2]{FM}.

\begin{theorem}
\label{thm:newrelativenodepolys}
The relative node polynomials $N_\delta(\alpha; \beta)$, for $\delta =
0, 1, 2, 3$ (resp., $\delta \le 6$) are as listed in
Appendix~\ref{app:relativenodepolys} (resp., as provided in the
ancillary files of this paper).
\end{theorem}

The polynomial
$N_\delta(\alpha; \beta)$ is of degree $3\delta$ by
Theorem~\ref{thm:relativenodepoly}. We compute the
terms of $N_\delta(\alpha; \beta)$ of degree $\ge 3 \delta - 2$.

\begin{theorem}
\label{thm:coefficients}
The terms of $N_\delta(\alpha; \beta)$ of (total) degree $\ge 3
\delta - 2$ are given by
\begin{displaymath}
\scriptsize
\begin{split}
N_\delta(\alpha; \beta) = & \frac{3^\delta}{\delta !}
\Bigg[ d^{2\delta} |\beta|^\delta 
 +\frac{\delta}{3} \Big[ -\frac{3}{2}(\delta-1)d^2-8 d |\beta| + 
|\beta|\alpha_1 + d \beta_1 + |\beta| \beta_1\Big] d^{2\delta-2}
|\beta|^{\delta-1} + \\
& + \frac{\delta }{9} \Big[
\frac{3 }{8} (\delta -1) (\delta-2)(3\delta - 1) d^4 + 12\delta
(\delta -1) d^3 |\beta| +
(11\delta + 1) d^2 |\beta|^2 -\frac{3}{2}\delta (\delta - 1) (d^3 \beta_1 + d^2 |\beta|
\alpha_1) +\\
&\quad -\frac{1}{2} (\delta + 5) (3 \delta - 2) d^2 |\beta| \beta_1 
 -8(\delta - 1)(d |\beta|^2 \alpha_1 + d |\beta|^2 \beta_1) 
+ \frac{1}{2}(\delta - 1) (d^2 \beta_1^2 + |\beta|^2 \alpha_1^2 +
|\beta|^2 \beta_1^2) +\\
&\quad +  (\delta -1)(d |\beta| \alpha_1 \beta_1 + d |\beta| \beta_1^2 +
 |\beta|^2 \alpha_1 \beta_1)
\Big] d^{2\delta - 4}
 |\beta|^{\delta - 2} + \cdots \Bigg],
\end{split}
\end{displaymath}
where $d = \sum_{i \ge 1} i (\alpha_i + \beta_i)$.
\end{theorem}

Theorem \ref{thm:coefficients} can
be extended to terms of $N_\delta(\alpha,
\beta) $ of degree $\ge 3 \delta - 7$ (see
Remark~\ref{rmk:morecoefficients}).
In Theorem \ref{thm:coefficients}, we observe that all coefficients of $N_\delta(\alpha; \beta)$ in
degree $\ge 3\delta - 2$ are of the form $\tfrac{3^\delta}{\delta !}$
times a polynomial in $\delta$. Without computing the coefficients, we can
extend this further. It is
conceivable to expect this property of the coefficients to hold for arbitrary degrees,
which, in the special case of non-relative Severi degrees, was shown by N.~Qviller~\cite{Qv10}.

\begin{proposition}
\label{prop:coefficientpolynomiality}
Every coefficient of $N_{\delta}(\alpha; \beta)$ in degree $\ge 3
\delta - 7$ is given, up to a factor of $\tfrac{3^\delta}{\delta !}$, by a polynomial in $\delta$ with rational coefficients.
\end{proposition}

In 1997, L.~G\"ottsche~\cite{Go} conjectured universal polynomiality for Severi
degrees of smooth polarized projective surfaces $(S, \L)$, where $\L$
is an ample line bundle on $S$. More
precisely, he conjectured, for any
fixed number of nodes, the existence of a universal polynomial
that, evaluated at Chern numbers of $(S, \L)$, equals the Severi degree of $(S, \L)$, provided that $\L$ is
sufficiently ample. G\"ottsche's conjecture was recently proved by the
celebrated work of Y.-J. Tzeng~\cite{Tz10}.

Our approach to planar enumerative geometry is combinatorial and inspired by
\emph{tropical geometry}, in which one replaces a
subvariety of a complex algebraic torus by a piecewise linear polyhedral
complex (see, for example, \cite{Ga06,RST05,SS04}). By the celebrated
Correspondence Theorem of G.~Mikhalkin~\cite[Theorem~1]{Mi03}, one can
replace the algebraic curve count in $\CC \PP^2$ by an enumeration
of certain \emph{tropical curves}. E.~Brugall\'e and
G.~Mikhalkin~\cite{BM1,BM2} introduced a class of decorated graphs, called
{\it (marked) floor diagrams} (see Section~\ref{sec:floordiagrams}),
which, if weighted correctly, are equinumerous to such tropical curves.
We use a version of these results that incorporates tangency
conditions due to S.~Fomin and G.~Mikhalkin~[Theorem~\ref{thm:relcorrespondence}]\cite{FM}. S.~Fomin and G.~Mikhalkin also introduced a
\emph{template decomposition} of floor diagrams, which we extend to
be suitable for the relative case. This decomposition is crucial in
the proofs of all results in this paper, as is the reformulation of
algebraic curve counts in terms of floor diagrams.

To the author's knowledge,  the polynomiality of the well-studied
relative Severi degrees $N^\delta_{\alpha, \beta}$ was not expected
and came as a surprise to many experts in the field. As the methods in
this paper are a natural, somewhat technical extension of those in
\cite{FB,FM},
 our contribution can also be seen as establishing
an unexpected result on enumerative geometry of plane curves, a field
with extensive history and the focus of immense recent study. 

In related work, F.~Ardila and the author~\cite{AB10} generalized the polynomiality of
Severi degrees to a family of (in general non-smooth) toric surfaces including $\CC
\PP^1\!$~$\times$~$\!\CC \PP^1$ and Hirzebruch surfaces. A main feature
is that we showed polynomiality not only in the multi-degree of the curves but also
``in the surface itself.''
In~\cite{BGM10}, A.~Gathmann, H.~Markwig and the author defined \emph{Psi-floor
diagrams} that enumerate
plane curves that satisfy point and tangency conditions, and
conditions given by \emph{Psi-classes}.
  We proved a Caporaso--Harris type recursion
 for Psi-floor diagrams, and showed that 
 \emph{relative descendant Gromov-Witten invariants} equal their tropical counterparts.


This paper is organized as follows. In Section \ref{sec:floordiagrams},
we review the definition of floor diagrams and their markings. In
Section \ref{sec:relativenodepolys}, we introduce a new decomposition
of floor diagrams compatible with tangency conditions. In Section~\ref{sec:proofs}, we
prove Theorems
\ref{thm:relativenodepoly}, \ref{thm:stability} and
\ref{thm:newrelativenodepolys} and Proposition~\ref{prop:variables}. In
Section~\ref{sec:relativecoefficients}, we prove Theorem~\ref{thm:coefficients} and
Proposition~\ref{prop:coefficientpolynomiality}.

\medskip

{\bf Acknowledgements.}
The author thanks the referees for helpful and careful comments and
suggestion that led to significant improvements of the article.
The author was partially supported by a Rackham One-Term Dissertation Fellowship and  by the NSF grant DMS-055588.

\section{Floor diagrams and relative markings}
\label{sec:floordiagrams}

Floor diagrams are a class of decorated graphs which, if weighted correctly, enumerate plane curves with
prescribed properties. They were 
introduced by E.~Brugall\'e and G.~Mikhalkin \cite{BM1,BM2} in the
non-relative case and generalized to the relative setting by S.~Fomin
and G.~Mikhalkin \cite{FM}. We begin with a review of the relative setup, following notation of \cite{FM}
(where floor diagrams
are called ``labeled floor diagrams'').

\begin{definition}
A \emph{floor diagram} $\D$ on a vertex set $\{1, \dots, d\}$ is a
directed graph (possibly with
multiple edges) with edge weights $w(e) \in \ZZ_{>0}$
satisfying:
\begin{enumerate}
\item The edge directions preserve the vertex order, i.e., for each
  edge $i \to j $ of $\D$ we have $i
  < j$.
\item (Divergence Condition) For each vertex $j$ of $\D$:
\[
\text{div}(j) \stackrel{\text{def}}{=} \sum_{ \tiny
     \begin{array}{c}
  \text{edges }e\\
j \stackrel{e}{\to} k
     \end{array}
} w(e) -   \sum_{ \tiny
     \begin{array}{c}
  \text{edges }e\\
i \stackrel{e}{\to} j
     \end{array}
} w(e)\le 1.
\]
\end{enumerate}
\vspace{-1mm}
This means that at every vertex of $\D$ the total weight of the outgoing edges
is larger by at most 1 than the total weight of the incoming edges.
\end{definition}

The \emph{degree} of a  floor
diagram $\D$ is the number of its vertices. $\D$ is
\emph{connected} if its underlying graph is. Note that in \cite{FM}
floor diagrams are required to be connected. If $\D$ is
connected its \emph{genus} 
is
the genus of the underlying graph
 (or the first Betti number of the underlying topological space). The \emph{cogenus}
of a connected  floor diagram $\D$
of degree $d$ and genus $g$ is given by
$\delta(\D) = \frac{(d-1)(d-2)}{2} - g$. If $\D$ is not connected, let
$d_1, d_2, \dots$ and $\delta_1, \delta_2, \dots$ be the degrees and
cogenera, respectively, of its connected components. The \emph{cogenus} of $\D$
is $\delta(\D) = \sum_j \delta_j + \sum_{j < j'} d_j d_{j'}$. Via the
correspondence between algebraic curves and floor diagrams
\cite[Theorem 2.5]{BM2}, these notions correspond literally to
the respective analogues for algebraic curves. Connectedness
corresponds to irreducibility. Lastly, a marked floor
diagram $\D$ has \emph{multiplicity}
 \footnote{This agrees with the multiplicity of a tropical plane curve
   degenerating to a particular floor diagram~\cite{FM}.}
\[
\mu(\D) \stackrel{\text{def}}{=} \prod_{\text{edges }e} w(e)^2.
\]

We draw floor diagrams using the convention that vertices in
increasing order are arranged left to right. Edge weights of $1$
are omitted.

\begin{example}
\label{ex:floordiagram}
An example of a floor diagram of degree $d = 4$, genus $ g=1$,
 cogenus $\delta = 2$, divergences $1,1,0,-2$, and multiplicity $\mu =
 4$ is drawn below.
\begin{center}
\begin{picture}(50,40)(30,-18)\setlength{\unitlength}{4pt}\thicklines
\oooo\Eeee\eEee\eeOe
\put(15,1.5){\makebox(0,0){$2$}} 
\put(7,0){\vector(1,0){1}} 
\put(17,0){\vector(1,0){1}}
\put(27.5,1.75){\vector(2,-1){1}}
\put(27.5,-1.75){\vector(2,1){1}}
\end{picture}
\end{center}
\end{example}

To enumerate algebraic curves satisfying tangency conditions, we need
the notion of marked floor diagrams.
Our notation, which is
more convenient for our purposes,
differs slightly from \cite{FM}, where S.~Fomin and G.~Mikhalkin define relative markings relative to the
partitions
$\lambda = \langle 1^{\alpha_1} 2^{\alpha_2} \cdots \rangle$
and $\rho = \langle 1^{\beta_1} 2^{\beta_2} \cdots \rangle$. In the
sequel, all sequences are sequences of non-negative
integers with finite support.

\begin{definition}
\label{def:relativemarking}
For two sequences $\alpha, \beta$ we define an \emph{$(\alpha,
  \beta)$-marking} of a  floor diagram~$\D$ of 
degree $d \! = \!\sum_{i \ge 1}\! i(\alpha_i \!+ \!\beta_i)$ by the following
four step process, which we illustrate in the case of Example~
\ref{ex:floordiagram} for $\alpha = (1, 0, 0, \dots)$ and $\beta =
(1,1,0,0,\dots)$.

{ \bf Step 1:} Fix a pair of collections of sequences $( \{ \alpha^i
\}, \{ \beta^i \} )$, where $i$ runs over the vertices of $\D$, such that:
\begin{enumerate}
\item The sums over each collection satisfy $\sum_{i = 1}^d \alpha^i =
  \alpha$ and  $\sum_{i =1}^d \beta^i =
  \beta$.
\item For all vertices $i$ of $\D$, we have $\sum_{j \ge 1} j(\alpha_j^i+ \beta_j^i)= 1 - \text{div}(i)$.
\end{enumerate} 
The second condition says that the ``degree of the pair $(\alpha^i,
\beta^i)$'' is compatible with the divergence at vertex $i$.
Each such pair $( \{ \alpha^i
\}, \{ \beta^i \} )$ is called \emph{compatible} with $\D$ and $(\alpha,
\beta)$. We omit writing down trailing zeros.

\begin{center}
\begin{picture}(50,58)(20,-40)\setlength{\unitlength}{4pt}\thicklines
\oooo\Eeee\eEee\eeOe
\put(15,1.5){\makebox(0,0){$2$}} 
\put(7,0){\vector(1,0){1}} 
\put(17,0){\vector(1,0){1}} 
\put(27.5,1.75){\vector(2,-1){1}}
\put(27.5,-1.75){\vector(2,1){1}}

\put(-7,-4){\makebox(0,0){$\alpha^i=$}}
\put(30,-4){\makebox(0,0){$(1)$}}
\put(-7,-8){\makebox(0,0){$\beta^i=$}}
\put(20,-8){\makebox(0,0){$(1)$}}
\put(30,-8){\makebox(0,0){$(0,1)$}}
\end{picture}
\end{center}

{ \bf Step 2:} For each vertex $i$ of $\D$ and every $j \ge 1$, create
$\beta_j^i$ new
vertices, called \emph{$\beta$-vertices} and illustrated as
\begin{picture}(4,0)(0,0)\setlength{\unitlength}{4pt}\thicklines
\put(1,1){\circle*{2}}
\end{picture}
, and connect them to $i$ with new edges of weight $j$
directed away from $i$.  For each vertex $i$ of $\D$ and every $j \ge 1$, create
$\alpha^i_j $ new
vertices, called \emph{$\alpha$-vertices} and illustrated as
\begin{picture}(5,0)(0,0)\setlength{\unitlength}{4pt}\thicklines
\put(1,1){\circle{2}}
\put(1,1){\circle*{1}}
\end{picture}
, and connect them to $i$ with new edges of weight~$j$
directed away from $i$.

\begin{center}
\begin{picture}(50,55)(40,-40)\setlength{\unitlength}{4pt}\thicklines
\oooo\Eeee\eEee\eeOe
\put(15,1.5){\makebox(0,0){$2$}} 
\put(7,0){\vector(1,0){1}} 
\put(17,0){\vector(1,0){1}} 
\put(27.5,1.75){\vector(2,-1){1}}
\put(27.5,-1.75){\vector(2,1){1}}

\put(-7,-4){\makebox(0,0){$\alpha^i=$}}
\put(30,-4){\makebox(0,0){$(1)$}}
\put(-7,-8){\makebox(0,0){$\beta^i=$}}
\put(20,-8){\makebox(0,0){$(1)$}}
\put(30,-8){\makebox(0,0){$(0,1)$}}

\put(20.7,0.7){\line(1,1){4}}
\put(20.7,0.7){\vector(1,1){3}}
\put(25,5){\circle*{2}}
\put(30.6,0.8){\line(1,1){4}}
\put(30.8,0.6){\line(2,1){8.3}}
\put(30.6,0.8){\vector(1,1){2.5}}
\put(30.8,0.6){\vector(2,1){5}}
\put(35,5){\circle*{2}}
\put(40,5){\circle{2}}
\put(40,5){\circle*{1}}
\put(31.5,4){\makebox(0,0){$2$}} 
\end{picture}
\end{center}

{ \bf Step 3:} Subdivide each edge of the original
 floor diagram $\D$ into two directed edges by introducing a new
vertex for each edge. The new edges inherit their weights and
orientations. Call the resulting graph $\tilde{\D}$.

\begin{center}
\begin{picture}(50,43)(45,-17)\setlength{\unitlength}{4pt}\thicklines
\oooo \Eeee \eEee \eeOe

\put(12.5,2){\makebox(0,0){$2$}}
\put(17.5,2){\makebox(0,0){$2$}}
\put(2.5,0){\vector(1,0){1}}
\put(7.5,0){\vector(1,0){1}}
\put(12.5,0){\vector(1,0){1}}
\put(17.5,0){\vector(1,0){1}}
\put(27.5,1.75){\vector(2,-1){1}}
\put(27.5,-1.75){\vector(2,1){1}}
\put(22.5,1.75){\vector(2,1){1}}
\put(22.5,-1.75){\vector(2,-1){1}}

\put(5,0){\circle*{2}}
\put(15,0){\circle*{2}}
\put(25,2.5){\circle*{2}}
\put(25,-2.5){\circle*{2}}

\put(20.7,0.7){\line(1,1){4}}
\put(20.7,0.7){\vector(1,1){3}}
\put(25,5){\circle*{2}}
\put(30.6,0.8){\line(1,1){4}}
\put(30.8,0.6){\line(2,1){8.3}}
\put(30.6,0.8){\vector(1,1){2.5}}
\put(30.8,0.6){\vector(2,1){5}}
\put(35,5){\circle*{2}}
\put(40,5){\circle{2}}
\put(40,5){\circle*{1}}
\put(31.5,4){\makebox(0,0){$2$}} 
\end{picture}
\end{center}

{ \bf Step 4:} Linearly order the vertices of $\tilde{\D}$ extending
the order of the vertices of the original floor diagram $\D$ such that, as in $\D$, each edge is directed from a
smaller vertex to a larger vertex.
Furthermore, we require that the $\alpha$-vertices are largest among
all vertices, and
for every pair of $\alpha$-vertices $i' > i$, the weight of the $i'$-adjacent
edge is larger than or equal to the weight of the $i$-adjacent edge.

\begin{center}
\begin{picture}(50,48)(75,-17)\setlength{\unitlength}{4pt}\thicklines
\put(12.5,2){\makebox(0,0){$2$}}
\put(17.5,2){\makebox(0,0){$2$}}
\multiput(0,0)(10,0){3}{\circle{2}}
\multiput(40,0)(10,0){1}{\circle{2}}
\multiput(5,0)(10,0){5}{\circle*{2}}
\multiput(30,0)(10,0){1}{\circle*{2}}
\put(50,0){\circle{2}}
\put(50,0){\circle*{1}}
\put(1,0){\line(1,0){8}}
\put(11,0){\line(1,0){8}}
\put(2.5,0){\vector(1,0){1}}
\put(7.5,0){\vector(1,0){1}}
\put(12.5,0){\vector(1,0){1}}
\put(17.5,0){\vector(1,0){1}}
\put(21,0){\line(1,0){3}}
\put(22.5,0){\vector(1,0){1}}
\qbezier(20.6,0.6)(22,6)(25,6)\qbezier(25,6)(28,6)(29.4,0.6)
\put(25,6){\vector(1,0){1}}
\qbezier(20.8,-0.6)(23.75,-3)(27.5,-3)\qbezier(27.5,-3)(32.25,-3)(34.2,-0.6)
\put(27.5,-3){\vector(1,0){1}}
\qbezier(25.8,0.6)(28.75,3)(32.5,3)\qbezier(32.5,3)(37.25,3)(39.2,0.6)
\put(32.5,3){\vector(1,0){1}}
\put(36,0){\line(1,0){3}}
\put(37.5,0){\vector(1,0){1}}
\put(41,0){\line(1,0){3}}
\put(42.5,0){\vector(1,0){1}}
\put(42.5,-2){\makebox(0,0){$2$}}
\qbezier(40.8,0.6)(42,3)(45,3)\qbezier(45,3)(48,3)(49.2,0.6)
\put(45,3){\vector(1,0){1}}

\end{picture}
\end{center}

We call the extended graph $\tilde{\D}$,
together with the linear order on its vertices, an \emph{$(\alpha,
  \beta)$-marked floor diagram}, or an \emph{$(\alpha, \beta)$-marking} of the floor diagram $\D$.
\end{definition}

We need to count $(\alpha,\beta)$-marked floor
diagrams up to equivalence. Two $(\alpha,\beta)$-markings
$\tilde{\D}_1$, $\tilde{\D}_2$ of a floor diagram $\D$ are \emph{equivalent} if there exists a weight preserving automorphism of weighted
graphs mapping $\tilde{\D}_1$ to $\tilde{\D}_2$ that fixes the vertices of $\D$. The \emph{number of
  markings} $\nu_{\alpha, \beta}(\D)$ is the number of
$(\alpha,\beta)$-markings of $\D$, up to equivalence. Furthermore, we write
$\mu_\beta(\D)$ for the product $1^{\beta_1} 2^{\beta_2}  \cdots
\mu(\D)$. In the example in
Definition~\ref{def:relativemarking}, the particular choice of
compatible sequences results in a choice of $5$ non-equivalent markings in Step~(4),
thus contributing $5$ to $\nu_{\alpha, \beta}(\D)$. 

The next
theorem follows from \cite[Theorem 3.18]{FM} by a straightforward extension of the
inclusion-exclusion procedure of \cite[Section 1]{FM} that was used
to conclude \cite[Corollary 1.9]{FM} (the non-relative count of reducible curves via
floor diagrams) from \cite[Theorem 1.6]{FM} (the non-relative count of irreducible
curves via floor diagrams).

\begin{theorem}
\label{thm:relcorrespondence}
For any $\delta \ge 1$, the
relative Severi degree $N^\delta_{\alpha, \beta}$ is given by
\begin{displaymath}
N^\delta_{\alpha, \beta} = \sum_\D \mu_\beta(\D) \nu_{\alpha, \beta}(\D),
\end{displaymath}
where the sum is over all (possibly disconnected) floor diagrams $\D$ of degree
$d = \sum_{i \ge 1} i (\alpha_i + \beta_i)$ and cogenus $\delta$.
\end{theorem}

\section{Relative Decomposition of Floor Diagrams} 
\label{sec:relativenodepolys}

In this section, we introduce a new decomposition of floor diagrams
compatible with tangency conditions, which we use extensively in Sections~\ref{sec:proofs}
and~\ref{sec:relativecoefficients} to prove all our results
stated in Section~\ref{sec:intro}.
This decomposition is a generalization of ideas of S.~Fomin and G.~Mikhalkin
\cite{FM}.
We start out by reviewing their key gadget.

\begin{definition}
\label{def:template}
A \emph{template} $\Gamma$ is a directed graph (possibly with multiple
edges) on vertices
$\{0, \dots, l\}$, where $l\ge 1$, with edge
weights $w(e) \in \ZZ_{>0}$, satisfying:
\begin{enumerate}
\item If $i \to j$ is an edge, then $i <j$.
\item Every edge $i \stackrel{e}{\to} i+1$ has weight $w(e) \ge
  2$. (No ``short edges.'')
\item For each vertex $j$, $1 \le j \le l-1$, there is an edge
  ``covering'' it, i.e., there exists an edge $i \to k$ with $i <j <k$.
\end{enumerate}
\end{definition}

Every template $\Gamma$ comes equipped with some natural numerical
invariants.
Its \emph{length} $l(\Gamma)$ is the number of
vertices minus $1$. The product of squares of the edge weights
is its \emph{multiplicity} $\mu(\Gamma)$. Its \emph{cogenus} $\delta(\Gamma)$
is
\[
\delta(\Gamma) \stackrel{\text{def}}{=} \sum_{\stackrel{e}{i \to j}} \bigg[(j-i) w(e) -1
\bigg].
\]
Thus, every edge of $\Gamma$ contributes by the product of its length
and weight minus $1$ to $\delta(\Gamma)$.

For $1 \le j \le l(\Gamma)$, let $\varkappa_j = \varkappa_j(\Gamma)$
denote the sum of the weights of edges $i \stackrel{e}{\to} k$ with $i
< j \le k$, i.e., $\varkappa_j(\Gamma)$ is the sum of the weights of
all edges of $\Gamma$ from a vertex left of or equal to $j-1$ to a
vertex right of or equal to $j$. We define
\[
k_{\min}(\Gamma) \stackrel{\text{def}}{=} \max_{1 \le j \le l}(\varkappa_j -j +1).
\]
This makes $k_{\min}(\Gamma)$ the smallest positive integer $k$ such
that $\Gamma$ can appear in a floor diagram on $\{1, 2, \dots \}$ as a
subgraph with
left-most vertex $k$.
Figure~\ref{fig:templates} (\cite[Figure 10]{FM}) lists all templates $\Gamma$ with
$\delta(\Gamma) \le 2$.

\begin{figure}[tbp]
\begin{center}
\begin{tabular}{c|c|c|c|c|c|c|c}
$\Gamma$ &
$\delta(\Gamma)$ & $\ell(\Gamma)$ 
& $\mu(\Gamma)$ 
& $\varkappa(\Gamma)$ 
& $k_{\min}(\Gamma)$ & $P_{\Gamma}(k)$ & $s(\Gamma)$
\\
\hline \hline
&&&&&&&\\[-.1in]
\begin{picture}(95,10)(-10,-4)\setlength{\unitlength}{2.5pt}\thicklines
\multiput(0,0)(10,0){2}{\circle{2}}
\put(5,2){\makebox(0,0){$\scriptstyle 2$}}
\Eeee
\end{picture}
& 1 & 1 & 4 & (2) & 2 & $k-1$ & 1
\\[.15in]
\begin{picture}(95,8)(-10,-4)\setlength{\unitlength}{2.5pt}\thicklines
\ooo
\qbezier(0.8,0.6)(10,5)(19.2,0.6)
\end{picture}
& 1 & 2 & 1 & (1,1) & 1 & $2k+1$ & 1
\\
\hline \hline
&&&&&&&\\[-.1in]
\begin{picture}(95,10)(-10,-4)\setlength{\unitlength}{2.5pt}\thicklines
\multiput(0,0)(10,0){2}{\circle{2}}
\put(5,2){\makebox(0,0){$\scriptstyle 3$}}
\Eeee
\end{picture}
& 2 & 1 & 9 & (3) & 3 & $k-2$ & 1
\\[.15in]
\begin{picture}(95,8)(-10,-4)\setlength{\unitlength}{2.5pt}\thicklines
\multiput(0,0)(10,0){2}{\circle{2}}
\put(5,3.5){\makebox(0,0){$\scriptstyle 2$}}
\put(5,-3.5){\makebox(0,0){$\scriptstyle 2$}}
\qbezier(0.8,0.6)(5,2)(9.2,0.6)
\qbezier(0.8,-0.6)(5,-2)(9.2,-0.6)
\end{picture}
& 2 & 1 & 16 & (4) & 4 & $\binom{k-2}{2}$ & 2
\\[.15in]
\begin{picture}(95,8)(-10,-4)\setlength{\unitlength}{2.5pt}\thicklines
\ooo
\qbezier(0.8,0.6)(10,4)(19.2,0.6)
\qbezier(0.8,-0.6)(10,-4)(19.2,-0.6)
\end{picture}
& 2 & 2 & 1 & (2,2) & 2 & $\binom{2k}{2}$ & 0
\\[.15in]
\begin{picture}(95,8)(-10,-4)\setlength{\unitlength}{2.5pt}\thicklines
\ooo
\qbezier(0.8,0.6)(10,4)(19.2,0.6)
\put(5,-2){\makebox(0,0){$\scriptstyle 2$}}
\Eeee
\end{picture}
& 2 & 2 & 4 & (3,1) & 3 & $2k(k-2)$ & 1
\\[.15in]
\begin{picture}(95,8)(-10,-4)\setlength{\unitlength}{2.5pt}\thicklines
\ooo
\qbezier(0.8,0.6)(10,4)(19.2,0.6)
\put(15,-2){\makebox(0,0){$\scriptstyle 2$}}
\eEee
\end{picture}
& 2 & 2 & 4 & (1,3) & 2 & $2k(k-1)$ & 1
\\[.15in]
\begin{picture}(95,8)(-10,-4)\setlength{\unitlength}{2.5pt}\thicklines
\oooo
\qbezier(0.8,0.6)(15,6)(29.2,0.6)
\end{picture}
& 2 & 3 & 1 & (1,1,1) & 1 & $3(k+1)$ & 0
\\[.15in]
\begin{picture}(95,8)(-10,-4)\setlength{\unitlength}{2.5pt}\thicklines
\oooo
\qbezier(0.8,0.6)(10,5)(19.2,0.6)
\qbezier(10.8,0.6)(20,5)(29.2,0.6)
\end{picture}
& 2 & 3 & 1 & (1,2,1) & 1 & $k(4k+5)$ & 0
\\[-.05in]
\end{tabular}
\end{center}
\caption{The templates with $\delta(\Gamma) \le 2$.}
\label{fig:templates}
\end{figure}

We now explain how to decompose a floor diagram $\D$ into a
collection of templates and further building blocks. The decomposition
depends on  tangency sequences $\alpha$ and $\beta$ as well as  
a pair $(\{\alpha^i \},
\{\beta^i \})$ compatible with $\D$ (see Step~1 of Definition~\ref{def:relativemarking}),
where $i$ runs over the vertices of $\D$.

Assume we are given such data $(\D, \{\alpha^i\}, \{\beta^i\})$, and let $d$ be the
degree of $\D$. We first construct two
(infinite) matrices $A$ and $B$: for $i \ge 1$, we define the $i$th
row $a_i$ resp.\ $b_i$ of $A$ resp.\ $B$ to be the sequence
$\alpha^{d-i}$ resp.\ $\beta^{d-i}$. (If $d-i \le 0$, i.e., if $i \ge
d$, we set $a_i = b_i = (0, 0, \dots)$.) This records the sequences
$\alpha^{d-1}$ resp.\ $\beta^{d-1}$ of the second to last vertex of $\D$
in the first row of $A$ resp.\ $B$, the sequences $\alpha^{d-2}$
resp.\ $\beta^{d-2}$ associated with the 
third to last vertex of $\D$ in the second row of $A$ resp.\ $B$, and
so on. 
Notice that we do not record in $A$ and $B$ the sequences $\alpha^d$ and $\beta^d$
associated with the right-most vertex of $\D$. These sequences satisfy
\begin{equation}
\label{eqn:non-negativity}
\alpha^d = \alpha - \sum_{i \ge 1} a_i \quad \text{and} \quad \beta^d = \beta - \sum_{i \ge 1} b_i
\end{equation}
and can thus be recovered from $\alpha$, $\beta$, $A$ and $B$.
Before we continue to describe the decomposition of a floor diagram
into templates, we illustrate the previous construction by an example. 

\begin{example}
\label{ex:alphabetadecomposition}
The pictured pair of sequences $(\{\alpha^i\}, \{ \beta^i \})$ (we
omit to write down zero-sequences), compatible with the floor diagram
$\D$ and $\alpha = (0,1)$, $\beta = (4,1)$,
\vspace{-2mm}
\begin{center}
\begin{picture}(50,64)(70,-40)\setlength{\unitlength}{3pt}\thicklines
\multiput(0,0)(10,0){8}{\circle{2}}
\qbezier(0.8,0.6)(4,5)(10,5)\qbezier(10,5)(16,5)(19.2,0.6)
\put(10,5){\vector(1,0){1}}
\put(11,0){\line(1,0){8}}
\put(15,0){\vector(1,0){1}}
\put(21,0){\line(1,0){8}}
\put(25.5,0){\vector(1,0){1}}
\qbezier(20.8,0.6)(22,3)(25,3)\qbezier(25,3)(28,3)(29.2,0.6)
\put(25.5,3){\vector(1,0){1}}
\qbezier(20.8,-0.6)(22,-3)(25,-3)\qbezier(25,-3)(28,-3)(29.2,-0.6)
\put(25.5,-3){\vector(1,0){1}}
\put(31,0){\line(1,0){8}}
\put(35,2){\makebox(0,0){$3$}}
\put(35,0){\vector(1,0){1}}

\qbezier(40.8,0.6)(42,4)(45,4)\qbezier(45,4)(48,4)(49.2,0.6)
\put(45.5,4){\vector(1,0){1}}
\qbezier(40.8,0.6)(42,1.5)(45,1.5)\qbezier(45,1.5)(48,1.5)(49.2,0.6)
\put(45.5,1.5){\vector(1,0){1}}
\qbezier(40.8,-0.6)(42,-1.5)(45,-1.5)\qbezier(45,-1.5)(48,-1.5)(49.2,-0.6)
\put(45.5,-1.5){\vector(1,0){1}}
\qbezier(40.8,-0.6)(42,-4)(45,-4)\qbezier(45,-4)(48,-4)(49.2,-0.6)
\put(45.5,-4){\vector(1,0){1}}

\put(51,0){\line(1,0){8}}
\put(55,2){\makebox(0,0){$2$}}
\put(55,0){\vector(1,0){1}}


\put(61,0){\line(1,0){8}}
\put(65.5,0){\vector(1,0){1}}
\qbezier(60.8,0.6)(62,2)(65,2)\qbezier(65,2)(68,2)(69.2,0.6)
\put(65.5,2){\vector(1,0){1}}
\qbezier(60.8,-0.6)(62,-2)(65,-2)\qbezier(65,-2)(68,-2)(69.2,-0.6)
\put(65.5,-2){\vector(1,0){1}}

 \put(-7,0){\makebox(0,0){$\D=$}}
 \put(-7,-6){\makebox(0,0){$\alpha^i=$}}
 \put(50,-6){\makebox(0,0){$(0,1)$}}
 \put(-7,-12){\makebox(0,0){$\beta^i=$}}
 \put(30,-12){\makebox(0,0){$(1)$}}
 \put(50,-12){\makebox(0,0){$(1)$}}
 \put(70,-12){\makebox(0,0){$(2,1)$}}
\end{picture}
\end{center}
\vspace{2mm}
determines the matrices 
\begin{center}
\vspace{-4mm}
\begin{displaymath}
\scriptsize
A = 
\begin{bmatrix}
\, \, 0  \, \,& \, \, 0 \, \,& \, \, 0\, \,&\cdots \\
0& 1 & 0 &\cdots\\
0 & 0 & 0 & \cdots \\
0 & 0 & 0 & \cdots \\
0 & 0 & 0 & \cdots \\
\vdots & \vdots & \vdots & \ddots\\
\end{bmatrix}
\quad  \text{\normalsize and }\quad
B = 
\begin{bmatrix}
\, \, 0 \, \,& \, \, 0 \, \,& \, \,  0 \, \,&\cdots \\
1& 0 & 0 &\cdots\\
0 & 0 & 0 & \cdots \\
1 & 0 & 0 & \cdots \\
0 & 0 & 0 & \cdots \\
\vdots & \vdots & \vdots & \ddots\\
\end{bmatrix}.
\end{displaymath}
\end{center}
\end{example}
\vspace{2mm}

Next, we describe how the triple $(\D, A, B)$, in turn, determines a collection
of templates, together with some extra data.
Let $l(A)$ resp.\ $l(B)$ be the largest row indices such
that $A$ resp.\ $B$ have a non-zero entry in this row. We call these numbers
the {\it length} of $A$ and $B$. (The length 
of the
zero-matrix is $0$.)  After we remove all ``short edges'' from $\D$, i.e., all
edges of weight $1$ between consecutive vertices,  the resulting
graph is an ordered collection of templates $(\Gamma_1,
\dots, \Gamma_r)$, listed left to right. Let $k_s$ be the smallest
vertex in $\D$ of each template $\Gamma_s$. Record all pairs
$(\Gamma_s,k_s)$ that satisfy $k_s + l(\Gamma_s) \le d -
\max(l(A),l(B))$, i.e., all templates whose right-most vertex is left
of or equal to every vertex $i$ of $\D$ for which $A$ or $B$ have row
$d-i$ non-zero.
Record the remaining templates, together with all
vertices $i$, for $i \ge \max(l(A), l(B))$, in \emph{one} graph
$\Lambda$ on vertices $0, \dots, l$ by
shifting the vertex labels by $d - l$. See Example
\ref{ex:relativedecomposition} for an example of this
decomposition.

\begin{example}
\label{ex:relativedecomposition}
The decomposition of the floor diagram $\D$ of
Example~\ref{ex:alphabetadecomposition} subject to the
matrices $A$ and $B$ of Example~\ref{ex:alphabetadecomposition} is
pictured below. Notice that $\max(l(A), l(B)) = 4$, thus $\Lambda$
contains both the edges of weight $2$ and $3$.
\begin{center}
\begin{picture}(50,80)(50,-85)\setlength{\unitlength}{3pt}\thicklines
\multiput(0,0)(10,0){8}{\circle{2}}
\qbezier(0.8,0.6)(4,5)(10,5)\qbezier(10,5)(16,5)(19.2,0.6)
\put(10,5){\vector(1,0){1}}
\put(11,0){\line(1,0){8}}
\put(15,0){\vector(1,0){1}}
\put(21,0){\line(1,0){8}}
\put(25.5,0){\vector(1,0){1}}
\qbezier(20.8,0.6)(22,3)(25,3)\qbezier(25,3)(28,3)(29.2,0.6)
\put(25.5,3){\vector(1,0){1}}
\qbezier(20.8,-0.6)(22,-3)(25,-3)\qbezier(25,-3)(28,-3)(29.2,-0.6)
\put(25.5,-3){\vector(1,0){1}}
\put(31,0){\line(1,0){8}}
\put(35,2){\makebox(0,0){$3$}}
\put(35,0){\vector(1,0){1}}

\qbezier(40.8,0.6)(42,4)(45,4)\qbezier(45,4)(48,4)(49.2,0.6)
\put(45.5,4){\vector(1,0){1}}
\qbezier(40.8,0.6)(42,1.5)(45,1.5)\qbezier(45,1.5)(48,1.5)(49.2,0.6)
\put(45.5,1.5){\vector(1,0){1}}
\qbezier(40.8,-0.6)(42,-1.5)(45,-1.5)\qbezier(45,-1.5)(48,-1.5)(49.2,-0.6)
\put(45.5,-1.5){\vector(1,0){1}}
\qbezier(40.8,-0.6)(42,-4)(45,-4)\qbezier(45,-4)(48,-4)(49.2,-0.6)
\put(45.5,-4){\vector(1,0){1}}

\put(51,0){\line(1,0){8}}
\put(55,2){\makebox(0,0){$2$}}
\put(55,0){\vector(1,0){1}}


\put(61,0){\line(1,0){8}}
\put(65.5,0){\vector(1,0){1}}
\qbezier(60.8,0.6)(62,2)(65,2)\qbezier(65,2)(68,2)(69.2,0.6)
\put(65.5,2){\vector(1,0){1}}
\qbezier(60.8,-0.6)(62,-2)(65,-2)\qbezier(65,-2)(68,-2)(69.2,-0.6)
\put(65.5,-2){\vector(1,0){1}}

\end{picture}
\begin{picture}(50,113)(103,-20)\setlength{\unitlength}{3pt}\thicklines
\put(35,12){\Huge $\downarrow$}
\multiput(0,0)(10,0){8}{\circle{2}}
\qbezier(0.8,0.6)(4,5)(10,5)\qbezier(10,5)(16,5)(19.2,0.6)
\put(31,0){\line(1,0){8}}
\put(35,2){\makebox(0,0){$3$}}
\put(35,0){\vector(1,0){1}}
\put(51,0){\line(1,0){8}}
\put(55,2){\makebox(0,0){$2$}}
\put(55,0){\vector(1,0){1}}


\put(10,-5){\makebox(0,0){$(\Gamma_1, 1)$}}
\put(50,-8){\makebox(0,0){$\Lambda$}}
\put(25,-5){\dashbox{2}(50,12) }
\end{picture}
\end{center}
\end{example}

The triple $(\Lambda, A, B)$ in the example above is an instance of an
``extended template,'' a new building block allowing relative
decomposition.

\begin{definition}
\label{def:extendedtemplate}
A tuple $(\Lambda, A, B)$ is an \emph{extended template} of
\emph{length} $l = l(\Lambda) = l(\Lambda, A, B)$ if $\Lambda$ is a
directed graph (possibly with multiple edges) on vertices $\{0, \dots, l\}$, where $l \ge 0$, with edge weights $w(e) \in
\ZZ_{>0}$, satisfying:
\begin{enumerate}
\item If $i \to j$ is an edge then $i <j$.
\item Every edge $i \stackrel{e}{\to} i+1$ has weight $w(e) \ge
  2$. (No ``short edges.'')
\end{enumerate}
Moreover, $A$ and $B$ are (infinite) matrices with
non-negative integral entries and finite support, and we write $l(A)$
and $l(B)$ for the respective largest row indices of $A$ and $B$ of a
non-zero entry. Additionally, we demand $ l(\Lambda) \ge
\max(\l(A), l(B))$, and that, for each $1 \le j \le l - \max(\l(A),
l(B))$, there is an edge $i \to k$ of $\Lambda$ with $i < j < k$.
\end{definition}

The last condition in the definition is a kind of ``connectedness''
property: it says that, for each vertex $j \ge 1$ of $\Lambda$
at least at distance $\max(\l(A),l(B))$ from the right-most vertex
$l$ of $\Lambda$, there must be an edge of $\Lambda$ that passes it. This
condition implies that once $l(Lambda) \ge 1$, the matrices $A$ and $B$
cannot both be the zero-matrix.

From a floor diagram $\D$, sequences $\alpha$ and $\beta$, and a
compatible pair $(\{ \alpha^i\}, \{ \beta^i\})$, we have constructed
two successive maps:
\begin{equation}
\label{eqn:relativedecomposition}
\big( \D, (\alpha^i), (\beta^i) \big) \longrightarrow \big( \D, A, B \big) \longrightarrow \big( \{(\Gamma_s,
k_s)\}, \Lambda, A, B \big).
\end{equation}
The two maps are illustrated in
Examples~\ref{ex:alphabetadecomposition}
and~\ref{ex:relativedecomposition}, respectively.

We now analyze when these maps are reversible.
Fix a collection $\{(\Gamma_s, k_s)\}$ of templates and positive
integers, where $1 \le s \le m$, an extended template
$(\Lambda, A, B)$, and a positive integer $d$ (which will be the
degree of $\D$). Then the second map of
(\ref{eqn:relativedecomposition}) is reversible if and only if
\begin{equation}
\label{eqn:inequalities}
\left\{
     \begin{array}{c}
       \begin{array}{rll}
k_i &\ge \, \, \, k_{\min}(\Gamma_i) \, &\text{ for } 1 \le i \le m, \\
k_{i+1}  & \ge \, \, \, k_i + l(\Gamma_i) \, & \text{ for } 1 \le i \le m-1, \\
k_m + l(\Gamma_m) & \le \, \, \ d - l(\Lambda). &
\end{array}
     \end{array}
   \right.
\end{equation}
In other words, the templates cannot appear too far to the left, and
the graphs $\Gamma_1, \dots, \Gamma_m$ and $\Lambda$ cannot overlap.

\begin{figure}[tbp]
\begin{center}
\begin{tabular}{c|c|c|c|c|c|c|c}
$(\Lambda, A, B)$ &
$\delta$ & $l$ 
& $\mu$ 
& $\varkappa$ 
& $\!\!d_{\min}\!\!$ & $q_{(\Lambda, A, B)}(\alpha; \beta)$ of Lemma~\ref{lem:polyoffinaltemplate}& $s$
\\
\hline \hline
&&&&&&&\\[-.1in]
\begin{picture}(60,10)(-0,-4)\setlength{\unitlength}{2.5pt}\thicklines
\o
\end{picture}
\scriptsize
$\begin{bmatrix}
0& 0 \\
0& 0\\
\end{bmatrix}
\begin{bmatrix}
0& 0 \\
0& 0\\
\end{bmatrix}$
& 0 & 0 & 1 & () & 1 & $1$ & 0
\\ [.1in]
\hline \hline
&&&&&&&\\ [-.1in]
\begin{picture}(60,10)(-0,-4)\setlength{\unitlength}{2.5pt}\thicklines
\oo
\end{picture}
\scriptsize
$\begin{bmatrix}
1& 0 \\
0& 0\\
\end{bmatrix}
\begin{bmatrix}
0& 0 \\
0& 0\\
\end{bmatrix}$
& 1 & 1 & 1 & (0) & 1 & $1$ & 0
\\
\begin{picture}(60,15)(-0,-4)\setlength{\unitlength}{2.5pt}\thicklines
\oo
\end{picture}
\scriptsize
$\begin{bmatrix}
0& 0 \\
0& 0\\
\end{bmatrix}
\begin{bmatrix}
1& 0 \\
0& 0\\
\end{bmatrix}$
& 1 & 1 & 1 & (0) & 1 & $\beta_1 (d + |\beta| -1)$ & 0 \\  [.1in]
\hline \hline
&&&&&&&\\[-.1in]
\begin{picture}(60,15)(-0,-4)\setlength{\unitlength}{2.5pt}\thicklines
\multiput(0,0)(10,0){2}{\circle{2}}
\put(5,2){\makebox(0,0){$\scriptstyle 2$}}
\Eeee
\end{picture}
\scriptsize
$\begin{bmatrix}
1& 0 \\
0& 0\\
\end{bmatrix}
\begin{bmatrix}
0& 0 \\
0& 0\\
\end{bmatrix}$
& 2 & 1 & 4 & (2) & 4 & $(d-3)$ & 1 \\
\begin{picture}(60,15)(-0,-4)\setlength{\unitlength}{2.5pt}\thicklines
\multiput(0,0)(10,0){2}{\circle{2}}
\put(5,2){\makebox(0,0){$\scriptstyle 2$}}
\Eeee
\end{picture}
\scriptsize
$\begin{bmatrix}
0& 0 \\
0& 0\\
\end{bmatrix}
\begin{bmatrix}
1& 0 \\
0& 0\\
\end{bmatrix}$
& 2 & 1 & 4 & (2) & 4 & $\beta_1 (d-3) (d + |\beta|-2)$ & 1 \\
\begin{picture}(60,15)(-0,-4)\setlength{\unitlength}{2.5pt}\thicklines
\ooo
\qbezier(0.8,0.6)(10,4)(19.2,0.6)
\end{picture}
\scriptsize
$\begin{bmatrix}
1& 0 \\
0& 0\\
\end{bmatrix}
\begin{bmatrix}
0& 0 \\
0& 0\\
\end{bmatrix}$
& 2 & 2 & 1 & \!\!(1,1) \!\!& 3 & $2 (d-2)$ & 0 \\
\begin{picture}(60,15)(-0,-4)\setlength{\unitlength}{2.5pt}\thicklines
\ooo
\qbezier(0.8,0.6)(10,4)(19.2,0.6)
\end{picture}
\scriptsize
$\begin{bmatrix}
0& 0 \\
0& 0\\
\end{bmatrix}
\begin{bmatrix}
1& 0 \\
0& 0\\
\end{bmatrix}$
& 2 & 2 & 1 & \!\! (1,1) \!\! & 3 & $\beta_1 (d-2) (2d + 2 |\beta|-3)$ & 0 \\
\begin{picture}(60,15)(-0,-4)\setlength{\unitlength}{2.5pt}\thicklines
\oo
\end{picture}
\scriptsize
$\begin{bmatrix}
2& 0 \\
0& 0\\
\end{bmatrix}
\begin{bmatrix}
0& 0 \\
0& 0\\
\end{bmatrix}$
& 2 & 1 & 1 & (0) & 3 & $1$ & 0 \\
\begin{picture}(60,15)(-0,-4)\setlength{\unitlength}{2.5pt}\thicklines
\oo
\end{picture}
\scriptsize
$\begin{bmatrix}
1& 0 \\
0& 0\\
\end{bmatrix}
\begin{bmatrix}
1& 0 \\
0& 0\\
\end{bmatrix}$
& 2 & 1 & 1 & (0) & 3 & $ \beta_1 (d + |\beta|-2)$ & 0 \\
\begin{picture}(60,15)(-0,-4)\setlength{\unitlength}{2.5pt}\thicklines
\oo
\end{picture}
\scriptsize
$\begin{bmatrix}
0& 0 \\
0& 0\\
\end{bmatrix}
\begin{bmatrix}
2& 0 \\
0& 0\\
\end{bmatrix}$
& 2 & 1 & 1 & (0) & 3 & {\scriptsize $\binom{\beta_1}{2} (d^2 +2 d
  |\beta| + |\beta|^2 -5 d -  5 |\beta|+6)$} & 0 \\
\begin{picture}(60,15)(-0,-4)\setlength{\unitlength}{2.5pt}\thicklines
\oo
\end{picture}
\scriptsize
$\begin{bmatrix}
0& 1 \\
0& 0\\
\end{bmatrix}
\begin{bmatrix}
0& 0 \\
0& 0\\
\end{bmatrix}$
& 2 & 1 & 1 & (0) & 3 & $1$ & 0 \\
\begin{picture}(60,15)(-0,-4)\setlength{\unitlength}{2.5pt}\thicklines
\oo
\end{picture}
\scriptsize
$\begin{bmatrix}
0& 0 \\
0& 0\\
\end{bmatrix}
\begin{bmatrix}
0& 1 \\
0& 0\\
\end{bmatrix}$
& 2 & 1 & 1 & (0) & 3 & $\beta_2 (|\beta|-1) (d + |\beta|-2)$ & 0 \\
\begin{picture}(60,15)(-0,-4)\setlength{\unitlength}{2.5pt}\thicklines
\ooo
\end{picture}
\scriptsize
$\begin{bmatrix}
0& 0 \\
1& 0\\
\end{bmatrix}
\begin{bmatrix}
0& 0 \\
0& 0\\
\end{bmatrix}$
& 2 & 3 & 1 &\!\! (0,0) \!\!& 3 & $1$ & 0 \\
\begin{picture}(60,15)(-0,-4)\setlength{\unitlength}{2.5pt}\thicklines
\ooo
\end{picture}
\scriptsize
$\begin{bmatrix}
0& 0 \\
0& 0\\
\end{bmatrix}
\begin{bmatrix}
0& 0 \\
1& 0\\
\end{bmatrix}$
& 2 & 3 & 1 & \!\! (0,0) \!\! & 3 & $\beta_1 (|\beta|-1) (2 d + |\beta|-3)$ & 0
\\[-.05in]
\end{tabular}
\end{center}
\caption{The extended templates with $\delta(\Lambda, A, B) \le 2$.}
\label{fig:extendedtemplates}
\end{figure}

Whether the first map is reversible depends on the sequences $\alpha$
and $\beta$. Recall that, given a compatible pair $(\{ \alpha^i \},
\{\beta^i\})$, we only recorded in the matrices $A$ and $B$ the pairs
$(\alpha^i, \beta^i)$ for $i < d$. The pair $(\alpha^d, \beta^d)$, in turn, is
determined by (\ref{eqn:non-negativity}). As the entries of $\alpha^d$
and $\beta^d$ represent numbers of edges we need to add at vertex $d$ to
obtain an $(\alpha, \beta)$-marking of $\D$ (see Step~2 of
Definition~\ref{def:relativemarking}), all entries of $\alpha^d$ and
$\beta^d$ need to be non-negative. Thus, the first map in
(\ref{eqn:relativedecomposition}) is reversible if and only if component-wise
\begin{equation}
\label{eqn:inequalities2}
\sum_{i \ge 1} a_i \le \alpha \quad \quad \text{ and } \quad \quad \sum_{i \ge 1} b_i \le \beta.
\end{equation}
It follows that, for large enough $\alpha$ and~$\beta$, the decomposition is
independent of $\alpha$ and $\beta$, given the part of the compatible
pair $(\{ \alpha^i\}, \{ \beta^i\})$, for $1 \le i <d$, away from the
right-most vertex of $\D$. It is exactly
this part that is recorded by the matrices $A$ and~$B$. Thus, we can
think of $A$ and $B$ as sort of ``placeholders'' for $\alpha^i$ and
$\beta^i$, in the case where $\alpha$ and $\beta$
varies.

\begin{example}
\label{ex:templatecomposition}
We now illustrate the reverse direction of
(\ref{eqn:relativedecomposition}), that is, how to build up a floor
diagram (and a compatible pair of sequences) from the pieces of the
decomposition. Let $\Gamma_1$, $\Gamma_2$ and $\Lambda$ to be the three graphs
and $A$ and $B$ be the two matrices below.
\begin{picture}(100,50)(-19,-20)\setlength{\unitlength}{3pt}\thicklines

\multiput(0,0)(10,0){2}{\circle{2}}
\put(1,0){\line(1,0){8}}
\put(5,2){\makebox(0,0){$2$}}
\put(5,0){\vector(1,0){1}}
\put(5,-4){\makebox(0,0){$\Gamma_1$}}

\put(0,0){
\multiput(30,0)(10,0){3}{\circle{2}}
\qbezier(30.8,0.6)(34,5)(40,5)
\qbezier(40,5)(46,5)(49.2,0.6)
\put(41,0){\line(1,0){8}}
\put(45,2){\makebox(0,0){$3$}}
\put(45,0){\vector(1,0){1}}
\put(40,-4){\makebox(0,0){$\Gamma_2$}}
} 

\put(0,0){
\multiput(70,0)(10,0){4}{\circle{2}}
\qbezier(80.8,0.6)(84,5)(90,5)
\qbezier(90,5)(96,5)(99.2,0.6)
\put(90,7){\makebox(0,0){$2$}}
\put(81,0){\line(1,0){8}}
\put(85,2){\makebox(0,0){$2$}}
\put(85,0){\vector(1,0){1}}
\put(85,-4){\makebox(0,0){$\Lambda$}}
} 
\end{picture}

\begin{center}
\vspace{-8mm}
\begin{displaymath}
\footnotesize
A = 
\begin{bmatrix}
\, \, 1 \, \,& \, \, 1 \, \,& \, \,  0 \, \,&\cdots \\ 
1 & 0 & 0 & \cdots \\
0 & 0 & 0 & \cdots \\
\vdots & \vdots & \vdots & \ddots\\
\end{bmatrix}
\quad \quad
B = 
\begin{bmatrix}
\, \, 0 \, \,& \, \, 0 \, \,& \, \,  0 \, \,&\cdots \\
1 & 0 & 0 & \cdots \\
0 & 1 & 0 & \cdots \\
0 & 0 & 0 & \cdots \\
\vdots & \vdots & \vdots & \ddots\\
\end{bmatrix}
\end{displaymath}
\vspace{-3mm}
\end{center}
Notice that, without the left-most vertex in $\Lambda$, the triple
$(\Lambda, A, B)$ is not an extended template (otherwise, 
$l(\Lambda) = 2$ although $l(B) = 3$, as $B$ has non-trivial third
row).

Let $\alpha$ and $\beta$ be sufficiently large tangency sequences satisfying
(\ref{eqn:inequalities2}).
The degree of
the floor diagram $\D$ we want to construct is $d = \sum_i i(\alpha_i + \beta_i)$ (following
B\'ezout's Theorem). We need to choose the
``positions'' $k_1$ and $k_2$ of the templates $\Gamma_1$ and
$\Gamma_2$ in accordance with (\ref{eqn:inequalities}). An invalid choice
is
$k_1 = 1$, as it violates $k_1 \ge k_{\text{min}}(\Gamma_1) = 2$. This
is reflecting the fact that the
divergence condition of floor diagrams forbids weight-$2$ edges
adjacent to the first vertex. 

A valid choice, however, is $k_1 = 2$ and $k_2=3$, as for $d$ large enough, $k_1$ and $k_2$
satisfy (\ref{eqn:inequalities2}). From $A$ and $B$ we can directly read off
${\alpha^i}$ and $\beta^i$ for $i<d$ and determine the
floor diagram $\D$: between each pair of adjacent vertices $i$ and
$i+1$, we need to add sufficiently many edges of
length $1$ and weight $1$ (the ``short edges''), so that, after
adding the $\alpha$- and $\beta$-edges of a marking of $\D$, the total
weight of the edges from vertices left of or equal to $i$ to vertices
right of or equal to $i+1$ is
$i$. This makes the divergence (i.e., ``outflow minus inflow'') of
each vertex of $\D$
in the marked floor diagram equal $1$. For a formula of the necessary
number of short edges see (\ref{eqn:wls}) and
(\ref{eqn:numberofshortedges}) below. To illustrate, if $(\D,
\{\alpha^i\}, \{ \beta^i \})$ is
\begin{center}
\begin{picture}(270,95)(5,-73)\setlength{\unitlength}{3pt}\thicklines
\multiput(0,0)(10,0){6}{\circle{2}}
\put(1,0){\line(1,0){8}}
 \put(6,0){\vector(1,0){1}}
 
\put(11,0){\line(1,0){8}}
 \put(15,3){\makebox(0,0){$2$}}
 \put(16,0){\vector(1,0){1}}
 \put(10,-3){\makebox(0,0){\small $k_1\! = \!2$}}

 \qbezier(20.8,0.6)(22,1.5)(25,1.5)\qbezier(25,1.5)(28,1.5)(29.2,0.6)
 \put(26,1.5){\vector(1,0){1}}

 \qbezier(20.8,-0.6)(22,-1.5)(25,-1.5)\qbezier(25,-1.5)(28,-1.5)(29.2,-0.6)
 \put(26,-1.5){\vector(1,0){1}}

 \qbezier(20.8,0.6)(24,5)(30,5)\qbezier(30,5)(36,5)(39.2,0.6)
 \put(31,5){\vector(1,0){1}}

 \put(31,0){\line(1,0){8}}
 \put(34,2.5){\makebox(0,0){$3$}}
 \put(36,0){\vector(1,0){1}}

 \put(20,-3){\makebox(0,0){\small $k_2\! = \!3$}}

\put(41,0){\line(1,0){8}}
 \put(46,0){\vector(1,0){1}}
 \qbezier(40.8,0.6)(42,1.5)(45,1.5)\qbezier(45,1.5)(48,1.5)(49.2,0.6)
 \put(46,1.5){\vector(1,0){1}}
 \qbezier(40.8,-0.6)(42,-1.5)(45,-1.5)\qbezier(45,-1.5)(48,-1.5)(49.2,-0.6)
 \put(46,-1.5){\vector(1,0){1}}
 \qbezier(40.8,0.6)(42,3)(45,3.5)\qbezier(45,3.5)(48,3.5)(49.2,0.6)
 \put(46,3.5){\vector(1,0){1}}
 \qbezier(40.8,-0.6)(42,-3)(45,-3.5)\qbezier(45,-3.5)(48,-3.5)(49.2,-0.6)
 \put(46,-3.5){\vector(1,0){1}}

 \qbezier(50.8,0.6)(52,2)(55,2)
 \qbezier(50.8,0.6)(52,3.5)(55,3.5)
 \qbezier(51,0.0)(52,0.5)(55,0.5)
 \qbezier(50.8,-0.6)(52,-2)(55,-2)
 \qbezier(50.8,-0.6)(52,-3.5)(55,-3.5)
 \qbezier(51,-0.0)(52,-0.5)(55,-0.5)
 \put(60,0){\makebox(0,0){{\huge $\dots$}}}

\multiput(70,0)(10,0){4}{\circle{2}}

\qbezier(65,2.5)(68,2.5)(69.2,0.6)
\qbezier(65,-2.5)(68,-2.5)(69.2,-0.6)
 \put(66,1.3){\makebox(0,0){\Large $\vdots$}}
 \put(66,-6){\makebox(0,0){\footnotesize $
 \begin{array}{c}
d-4 \\
\text{many}
\end{array}
$}}

\qbezier(70.8,0.6)(72,2.5)(75,2.5)\qbezier(75,2.5)(78,2.5)(79.2,0.6)
\qbezier(70.8,-0.6)(72,-2.5)(75,-2.5)\qbezier(75,-2.5)(78,-2.5)(79.2,-0.6)
 \put(75,1.3){\makebox(0,0){\Large $\vdots$}}
 \put(75,-6){\makebox(0,0){\footnotesize $
 \begin{array}{c}
d-5 \\
\text{many}
\end{array}
$}}

\qbezier(80.8,-0.6)(82,-1)(85,-1)\qbezier(85,-1)(88,-1)(89.2,-0.6)
\qbezier(80.8,-0.6)(82,-5)(85,-5)\qbezier(85,-5)(88,-5)(89.2,-0.6)
\put(85,-1.9){\makebox(0,0){\Large $\vdots$}}
 \put(85,-9){\makebox(0,0){\footnotesize $
 \begin{array}{c}
d-9 \\
\text{many}
\end{array}
$}}

\qbezier(90.8,0.6)(92,2.5)(95,2.5)\qbezier(95,2.5)(98,2.5)(99.2,0.6)
\qbezier(90.8,-0.6)(92,-2.5)(95,-2.5)\qbezier(95,-2.5)(98,-2.5)(99.2,-0.6)
 \put(95,1.3){\makebox(0,0){\Large $\vdots$}}
 \put(95,-6){\makebox(0,0){\footnotesize $
 \begin{array}{c}
d-9 \\
\text{many}
\end{array}
$}}

\put(81,0){\line(1,0){8}}
 \put(86,2.5){\makebox(0,0){$2$}}
 \put(86,0){\vector(1,0){1}}

 \qbezier(80.8,0.6)(84,5)(90,5)\qbezier(90,5)(96,5)(99.2,0.6)
 \put(91,5){\vector(1,0){1}}

 \put(-5,-3){\makebox(0,0){$\D=$}}
 \put(-5,-17){\makebox(0,0){$\alpha^i=$}}
 \put(80,-17){\makebox(0,0){$(1)$}}
 \put(90,-17){\makebox(0,0){$(1,1)$}}
 \put(100,-17){\makebox(0,0){$\alpha^d$}}
  \put(-5,-23){\makebox(0,0){$\beta^i=$}}
  \put(70,-23){\makebox(0,0){$(0,1)$}}
  \put(80,-23){\makebox(0,0){$(1)$}}
 \put(100,-23){\makebox(0,0){$\beta^d$}}
\end{picture}
\end{center}
we need to add $d-5$ short edges between the third-last and
second-last vertex of~$\D$: each marking $\tilde{\D}$ of $\D$ has one
$\beta$-edge of weight $2$ emerging from the third-last vertex (as
$\beta^{d-3} = (0,1)$), and we want the total weight of edges of $\tilde{\D}$ between vertices left of or equal to vertex $d-3$ and right of or
equal to vertex $d-2$ to equal $d-3$.

The sequences $\alpha^d$ and $\beta^d$ of the right-most vertex of
$\D$ are functions of the tangency sequences $\alpha$ and $\beta$. By (\ref{eqn:non-negativity}), we have $\alpha^d = (\alpha_1-2,
\alpha_2-1, \alpha_3, \alpha_4, \dots)$
and $\beta^d = (\beta_1, \beta_2 - 1, \beta_3 - 1, \beta_4,
\dots)$. Thus, the inverse map of (\ref{eqn:relativedecomposition}) is
defined provided $\alpha_1 \ge 2$, $\alpha_2, \beta_2, \beta_3 \ge 1$, and
$\beta_1, \alpha_3, \alpha_i, \beta_i \ge 0$, for $i \ge 4$. This concludes the example.
\end{example}

The \emph{cogenus} $\delta(\Lambda, A, B)$ of an extended template $(\Lambda,
A, B)$ is
the sum of the cogenera $\delta(\Lambda)$, $\delta(A)$ and $\delta(B)$, where
\[
\delta(\Lambda) \stackrel{\text{def}}{=} \sum_{\stackrel{e}{i \to j}} \Big((j-i) w(e) -1
\Big)) \quad \text{ and } \quad \delta(A) \stackrel{\text{def}}{=} \sum_{i, j \ge 1} i \cdot
j \cdot a_{i,j},
\]
and similarly for $B$.
Figure~\ref{fig:extendedtemplates} shows all extended templates with
cogenus at most $2$. This list can be obtained by first considering
all pairs of matrices $(A, B)$ with $\delta(A) + \delta(B) \le 2$, and
then possibly adding weighted edges such that the last (``connectedness'') condition in
Definition~\ref{def:extendedtemplate} is still satisfied.
The definition of $\delta(\Lambda, A, B)$ is such that the correspondence
(\ref{eqn:relativedecomposition}) is cogenus-preserving, in the
following sense.

\begin{proposition}
Fix two tangency sequences $\alpha$ and $\beta$. Let $\D$ be a floor
diagram of degree $d = \sum_{i \ge 1} i(\alpha_i + \beta_i)$. Then we
have, for each pair $(\{
\alpha^i\}, \{ \beta^i\})$ compatible with $\D$ and $(\alpha, \beta)$,
\begin{displaymath}
\delta(\D) = \left( \sum_{i = 1}^m \delta(\Gamma_i) \right) + \delta(\Lambda) +
\delta(A) + \delta(B),
\end{displaymath}
where $\Gamma_1, \dots, \Gamma_m$, $\Lambda$, $A$ and $B$ are defined
via (\ref{eqn:relativedecomposition}).
\end{proposition}

\begin{proof}
First, assume that $\D$ is connected with genus $g(\D)$. Then, by definition, $\delta(\D) =
\tfrac{(d-1)(d-2)}{2}-g(\D)$ is the ``genus-deficiency'' of
$\D$, with respect to the unique template $\D_0$ of degree~$d$ and
genus $\tfrac{(d-1)(d-2)}{2}$ (so $\delta(\D_0) =
0$). Specifically, $\D_0$ has $i$ edges of weight $1$ between vertices
$i$ and $i+1$, for $1 \le i < d$, and no other edges. Let $\Gamma_1,
\dots, \Gamma_m, \Lambda, A$ and $B$ be the decomposition data of $\D$
according to (\ref{eqn:relativedecomposition}). We will describe a
degeneration of $\D$ to $\D_0$ which, at each stage, is cogenus-preserving.

If $\Lambda$ has an edge, say $i \stackrel{e}{\to} j$, let $\Lambda'$ be
obtained from $\Lambda$ by removing the edge $e$. The floor diagram
$\D'$ corresponding to the altered data, with $\Lambda'$ instead of
$\Lambda$, differs from $\D$ by $\wt(e)$ many short edges between each
adjacent pair of the $\len(e)+1$ many adjacent vertices in $\D'$
between $i$ and $j$, where
$\len(e) = j - i$ is the length of $e$. This
alteration increases the genus by $\wt(e) \len(e) - 1$, which agrees
with the weight of the edge $e$ in the definition of
$\delta(\Lambda)$. Thus, we can assume that $\Lambda$ has no edges,
and, by the same argument applied to the templates $\Gamma_i$, that $m
= 0$.

If $B$ has a non-zero entry $b_{i'j'}$,  let $B'$ be the matrix
obtained from $B$ by lowering its
$(i',j')$th entry by one. Again, let $\D'$ be the floor diagram
that corresponds to the altered data. Then $\D'$ differs from $\D$ by $j'$
additional short edges between each adjacent pair of vertices in $\D'$
between $d-i'$ and $d$, increasing the genus by $i'\cdot j'$. This
agrees with the weight of the $(i', j')$th entry of $B$ in the
definition of $\delta(B)$, and we can assume that $B$ is
zero. Similarly, we can assume that $A$ is zero. Thus, if $\D$ is
connected, we are done as the proposition holds for $\D_0$.

If $\D$ equals a union $\D_1 \sqcup \D_2$ of two connected floor
diagrams $\D_1$ and $\D_2$, let $\tilde{\D}$ be an $(\alpha, \beta)$-marking of
$\D$. To prove the proposition for  non-connected floor diagrams, we
use the correspondence between marked floor diagrams and
  tropical plane curves\footnote{It is possible, although very tedious, to prove this lemma
  for disconnected floor diagrams $\D$ purely combinatorially. In the
  interest of a more compact presentation, we chose the
  route via tropical geometry.}
 through a fixed ``horizontally stretched'' point
  configuration\footnote{A configuration $\{ (x_i, y_i)\}$ of
    $\tfrac{(d+3)d}{2} - \delta$ points in
    $\RR^2$ is \emph{horizontally stretched} if, for all $i$, $x_i < x_{i+1}$, $y_i
    < y_{i+1}$,  and $\min_{i \neq j} |x_i - x_j| > (d^3 + d) \cdot \max_{i
      \neq j} |y_i - y_j|$.}.
All relevant definitions can, for example, be found in
\cite[Sections~2 and~3]{FM}. For an illustration of the correspondence,
see Figure~\ref{fig:marking_tropicalcurve_correspondence}.

 We can identify the marked floor diagram $\tilde{\D}$ with a tropical
 plane curve $C$
through a horizontally stretched point configuration in $\RR^2$
 (see \cite[Theorem~3.17]{FM}).
The two connected components $\tilde{\D_1}$ and $\tilde{\D_2}$ of
$\tilde{\D}$ are markings of $\D_1$ and $\D_2$, respectively. Each $\tilde{\D_i}$
corresponds to an irreducible component $C_i$ of $C$ (i.e., each $C_i$
is not a non-trivial union of tropical 
plane curves).

\begin{figure}
\begin{picture}(100,130)(120,5)
\put(30,130){\reflectbox{\includegraphics[scale=0.5,
    angle=270]{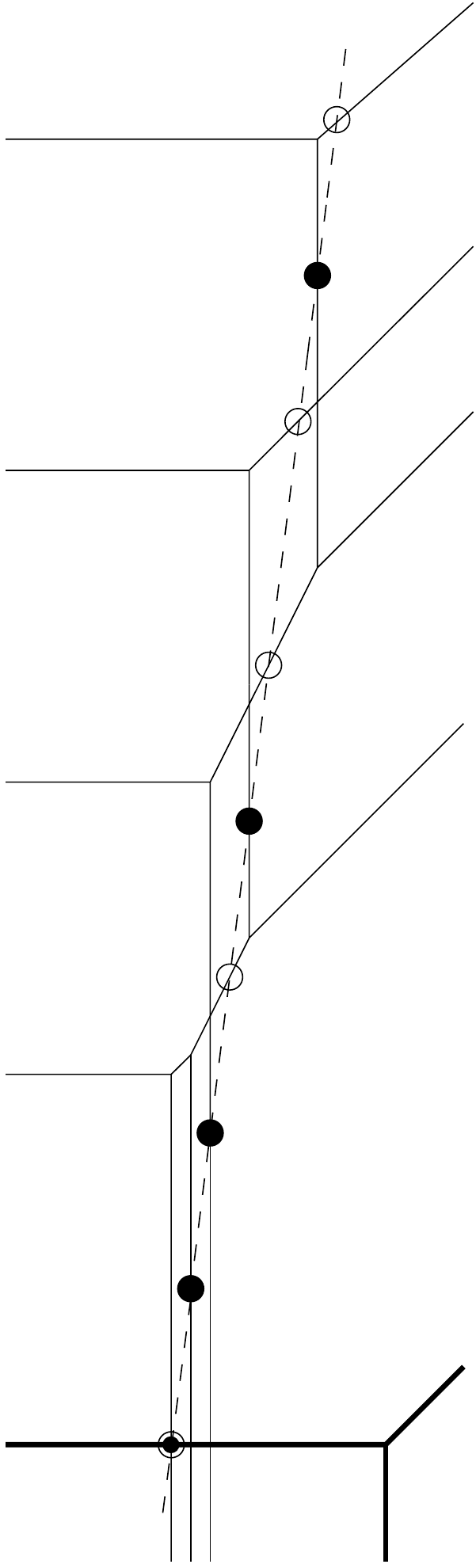}}}
\put(193,94){$2$}
\put(98,68){\scalebox{1.5}{$\Box$}}
\put(153,81){\scalebox{1.5}{$\Box$}}
\put(211,88){\scalebox{1.5}{$\Box$}}
\put(30,0){
\begin{picture}(50,35)(-55,-23)\setlength{\unitlength}{4pt}\thicklines
\oooo 
\put(1,0){\line(1,0){3}}
\put(2.5,0){\vector(1,0){1}}
\put(5,0){\circle*{2}}

\put(26,0){\line(1,0){3}}
\put(27.5,0){\vector(1,0){1}}
\put(25,0){\circle*{2}}
\put(35,0){\circle*{2}}
\put(40,0){\circle*{2}}
\put(45,0){\circle*{1}}
\put(45,0){\circle{2}}

\qbezier(10.8,0.6)(14,2.5)(17.5,2.5)
\qbezier(17.5,2.5)(21,2.5)(24.2,0.6)
\put(17.5,2.5){\vector(1,0){1}}

\qbezier(5.8,-0.6)(9,-2.5)(12.5,-2.5)
\qbezier(12.5,-2.5)(16,-2.5)(19.2,-0.6)
\put(12.5,-2.5){\vector(1,0){1}}

\qbezier(20.8,-0.6)(24,-2.5)(27.5,-2.5)
\qbezier(27.5,-2.5)(31,-2.5)(34.2,-0.6)
\put(27.5,-2.5){\vector(1,0){1}}
\put(27.5,-4){\makebox(0,0){$2$}}

\qbezier(30.8,0.6)(32,1.5)(35,1.5)
\qbezier(35,1.5)(38,1.5)(39.2,0.6)
\put(36,1.5){\vector(1,0){1}}
\qbezier(30.8,0.6)(33,3)(37.5,3)
\qbezier(37.5,3)(42,3)(44.2,0.6)
\put(38.5,3){\vector(1,0){1}}
\end{picture}

} 
\end{picture}
\caption{Correspondence between connected components of a relative
  marked floor diagram and irreducible components of a tropical
  plane curve through a horizontally stretched configuration (on a
  dashed line), tangent
  to a tropical line (in bold). The three points of intersection  of the
  components of the tropical curves are highlighted by $\Box$. The
  right-most such intersection is of multiplicity $a = 2$, the two
  others
 of multiplicity $a = 1$. We omit writing weights equal
  to $1$.}
\label{fig:marking_tropicalcurve_correspondence}
\end{figure}

 This identification is degree- and
 genus-preserving, thus $d(\D_j)$ equals the degree\footnote{For the
   purpose of this proof, we only need that there is well defined
   notion of the degree $d(C)$ of a tropical plane curve $C$, as well as a notion
 of the number $\delta(C)$ of its nodes (again, see \cite[Sections~2 and~3]{FM}.} $d(C_i)$ of the
 tropical curve $C_i$ and $\delta(\D_i)$ equals the number
 $\delta(C_i)$ of tropical nodes of~$C_i$. As the tropical curve $C$
 passes through a horizontally stretched point configuration, all
 intersection points of $C_1$ with $C_2$ are locally intersections of a
 horizontal edge with some weight $a \in \ZZ_{\ge 1}$ and an edge
 of slope $\tfrac{1}{n}$, for $n \in \ZZ_{\ge 1}$, and weight $1$
 \cite[Section 5]{BM2} (see Figure~\ref{fig:tropical_node}).
Such intersection is of multiplicity
 $a$ (see, for example, \cite[Theorem~4.2]{RST05}) and thus
 contributes $a$ to the number of nodes of $C$. 

\begin{figure}[h]
\begin{center}
\begin{picture}(60,15)(0,40) \setlength{\unitlength}{1.5pt}\thicklines
\put(10,30){\line(1,0){40}}
\put(14,22){\line(2,1){32}}
\put(13,33){\large a}
\put(37,37){\large 1}
\end{picture}
\end{center}
\label{fig:tropical_node}
\caption{The intersection of a horizontal
  weight $a$ edge and a non-horizontal weight~$1$ edge.}
\end{figure}
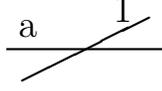

We claim that the number of nodes of $C$ is
\begin{equation}
\label{eqn:totalnodenumber}
\delta(C) = \big( \sum_{j \in J} \delta(\Gamma_j) \big) +
\delta(\Lambda, A, B).
\end{equation}
Indeed, the nodes of $C$ come in two types: firstly, bounded horizontal edges $e$ of weight
at least $2$ contribute $\wt(e) - 1$ nodes each ($e$ represents $\wt(e)$
many identified weight-$1$ edges, resulting in a genus deficiently of
$\wt(e) -1$). Secondly, intersections of
the form as in Figure~\ref{fig:tropical_node} contribute $a$
nodes each (cf.\ with \cite[Theorem~4.2]{RST05}). The former correspond to the edges of $\D$ of weight at least $2$.
The latter correspond to pairs $(e, j)$, where $j$ is a vertex of $\D$
and $e$ is either an edge $i \to k$ of $\D$ of weight $a$, for some $a
\in \ZZ_{\ge 1}$ and $i < j
< k$, or an $\alpha$- or $\beta$-edge of the corresponding marking of
$\D$ with source $i < j$. Thus, the bounded edges $e$, recorded in
$\Gamma_1, \dots, \Gamma_m$, and $\Lambda$, contribute $\wt(e)\len(e) - 1$
to the number of nodes of $C$. Each of the $a_{ij}$ resp.\ $b_{ij}$ many
$\alpha$- resp. $\beta$-edges of weight $j$ with source $d+1-i$ contribute $j$
nodes. Thus (\ref{eqn:totalnodenumber}) follows by the definition of
$\delta(\Gamma_i)$, $\delta(\Lambda)$, $\delta(A)$, and $\delta(B)$.

By the tropical B\'ezout Theorem \cite[Theorem~4.2]{RST05}, the
tropical curves $C_1$ and
$C_2$ intersect in $d(C_1) d(C_2) = d(\D_1)d(\D_2)$ many nodes. Thus
we have
\begin{displaymath}
\label{eqn:node_difference}
\big( \sum_{j \in J} \delta(\Gamma_j) \big) +
\delta(\Lambda, A, B) - d(C_1)  \cdot d(C_2)
= \delta(C_1) + \delta(C_2) = \delta(\D_1) + \delta(\D_2).
\end{displaymath}
For floor diagrams $\D$ with two components, the proposition follows from
the definition of the cogenus $\delta(\D)$. The proof for more than
two components is similar.

\end{proof}



With an extended template $(\Lambda, A, B)$ we further associate the following
numerical data:
for $1 \le j \le l(\Lambda)$, let $\varkappa_j(\Lambda) $
denote the sum of the weights of edges $i \to k$ of
$\Lambda$ with $i
< j \le k$. Define $d_{\min}(\Lambda, A, B)$ to be the smallest positive integer $d$ such
that $(\Lambda, A, B)$ can appear (at the right end) in a floor
diagram on $\{1, 2, \dots, d \}$. We will see later that
$d_{\min}$ is given by an explicit formula.
For a matrix $A = (a_{ij})$ of non-negative integers with finite support
define the ``weighted lower sum sequence''~$\wls(A)$ by
\begin{equation}
\label{eqn:wls}
\wls(A)_i \stackrel{\text{def}}{=} \sum_{
i' \ge i, 
j \ge 1
} j \cdot a_{i'j}.
\end{equation}
This sequence records, for each row $i$ of $A$, the sum of entries of
$A$ in or below the $i$th row of $A$, weighted by the column
index. As can be seen from Step~2 of
Definition~\ref{def:relativemarking}, $\wls(A)_i\,$ resp.\ $\wls(B)_i \,$
equals the total weight of $\alpha$-edges resp.\ $\beta$-edges that
pass vertex $d(\D) - i + 1$ in an $(\alpha, \beta)$-marking of a
floor diagram $\D$ ($A$ and
$B$ are obtained from this marking via
(\ref{eqn:relativedecomposition})).

We now define the number of ``markings'' of templates and
extended templates and relate them to the number of $(\alpha,
\beta)$-markings of the corresponding floor diagrams. To each template
$\Gamma$, we associate a polynomial: for $k \ge k_{\min}(\Gamma)$, let $\Gamma_{(k)}$ denote
the graph obtained from $\Gamma$ by first adding $k+i-1-\varkappa_i$
short edges connecting $i-1$ to i, for $1 \le i \le l(\Gamma)$, and
then subdividing each edge of the resulting graph by introducing one
new vertex for each edge. The number of short edges in the first step
equals the number of edges removed during the template
decomposition. For example, if $\Gamma$ is the second template from
the top in Figure~\ref{fig:templates}, then $\Gamma_{(k)}$ has $k-1$
resp.\ $k$ (subdivided) edges between the first and second resp.\
second and third vertex of~$\Gamma$. By \cite[Lemma 5.6]{FM} the number of
linear extensions (up to equivalence, see the paragraph after Definition \ref{def:relativemarking}) of
the vertex poset of the graph $\Gamma_{(k)}$ extending the vertex
order of $\Gamma$ is given by a polynomial $P_{\Gamma}(k)$ in $k$,
whenever $k \ge k_{\min}(\Gamma)$ (see Figure~\ref{fig:templates}).

For each pair of sequences $(\alpha, \beta)$ and each extended template $(\Lambda, A,
B)$ satisfying~(\ref{eqn:inequalities2}) and $d \ge d_{\min}$, where $d = \sum_{i \ge
  1} i (\alpha_i + \beta_i)$, we define its ``number of markings'' as follows. Write $l =l(\Lambda)$ and let
$\P(\Lambda, A, B)$ be the poset obtained from $\Lambda$ by
\begin{enumerate}
\item first creating an additional vertex $l + 1$ ($> l$),
\item
\label{itm:beta-edges}
then adding $b_{ij}$ edges of weight $j$ between $l  - i$
  and $l + 1$, for $1 \le i \le l$ and $j \ge 1$,
\item then adding $\beta_j - \sum_{i\ge 1} b_{ij}$ edges of weight $j$ between $l$
  and $l + 1$, for $j \ge 1$,
\item then adding
\begin{equation}
\label{eqn:numberofshortedges}
d - l(\Lambda) + i - 1 - \varkappa_i(\Lambda) - \wls(A)_{l + 1 -
  i} - \wls(B)_{l+1 - i}
\end{equation}
(``short'') edges of weight $1$ connecting $i-1$ and $i$, for $1 \le i \le l$, and
finally
\item subdividing all edges of the resulting graph by introducing a
  midpoint vertex for each edge.
\end{enumerate}

\begin{example}
If $(\Lambda, A, B)$ is the extended template of
Example~\ref{ex:templatecomposition}, then $\P(\Lambda, A, B)$ is
(edges are ordered left to right)
\begin{center}
\begin{picture}(270,120)(195,-80)\setlength{\unitlength}{4pt}\thicklines

\put(57,0){\makebox(0,0){$\P(\Lambda, A, B) = $}}

\multiput(70,0)(10,0){5}{\circle{2}}

\qbezier(70.8,0.6)(72,2.5)(75,2.5)\qbezier(75,2.5)(78,2.5)(79.2,0.6)
\qbezier(70.8,-0.6)(72,-2.5)(75,-2.5)\qbezier(75,-2.5)(78,-2.5)(79.2,-0.6)
 \put(75,1.3){\makebox(0,0){\Large $\vdots$}}
 \put(75,-6){\makebox(0,0){\footnotesize $
 \begin{array}{c}
d-5 \\
\text{many}
\end{array}
$}}
\put(75,2.5){\circle*{1.3}}
\put(75,-2.5){\circle*{1.3}}

\qbezier(80.8,-0.6)(82,-1)(85,-1)\qbezier(85,-1)(88,-1)(89.2,-0.6)
\qbezier(80.8,-0.6)(82,-5)(85,-5)\qbezier(85,-5)(88,-5)(89.2,-0.6)
\put(85,-1.9){\makebox(0,0){\Large $\vdots$}}
 \put(85,-9){\makebox(0,0){\footnotesize $
 \begin{array}{c}
d-9 \\
\text{many}
\end{array}
$}}
\put(85,-1){\circle*{1.3}}
\put(85,-5){\circle*{1.3}}

\qbezier(90.8,0.6)(92,2.5)(95,2.5)\qbezier(95,2.5)(98,2.5)(99.2,0.6)
\qbezier(90.8,-0.6)(92,-2.5)(95,-2.5)\qbezier(95,-2.5)(98,-2.5)(99.2,-0.6)
 \put(95,1.3){\makebox(0,0){\Large $\vdots$}}
 \put(95,-6){\makebox(0,0){\footnotesize $
 \begin{array}{c}
d-9 \\
\text{many}
\end{array}
$}}
\put(95,-2.5){\circle*{1.3}}
\put(95,2.5){\circle*{1.3}}

\put(81,0){\line(1,0){8}}
 \put(83.5,1.5){\makebox(0,0){$2$}}
 \put(87,1.5){\makebox(0,0){$2$}}

\put(85,0){\circle*{1.3}}

 \qbezier(80.8,0.6)(84,5)(90,5)\qbezier(90,5)(96,5)(99.2,0.6)
\put(90,5){\circle*{1.3}}

\qbezier(80.8,0.6)(86,6.5)(95,6.5)\qbezier(95,6.5)(104,6.5)(109.2,0.6)
\put(95,6.5){\circle*{1.3}}

\qbezier(70.8,0.6)(78,8)(90,8)\qbezier(90,8)(102,8)(109.2,0.6)
 \put(79,8){\makebox(0,0){$2$}}
 \put(101,8){\makebox(0,0){$2$}}
\put(90,8){\circle*{1.3}}

\qbezier(100.8,0.6)(102,2.5)(105,2.5)\qbezier(105,2.5)(108,2.5)(109.2,0.6)
\put(105,2.5){\circle*{1.3}}

\qbezier(100.8,-0.6)(102,-2.5)(105,-2.5)\qbezier(105,-2.5)(108,-2.5)(109.2,-0.6)
\put(105,-2.5){\circle*{1.3}}

 \put(105,1.3){\makebox(0,0){\Large $\vdots$}}
 \put(105,-14){\makebox(0,0){\footnotesize $
 \begin{array}{c}
\beta_1 \text{ of wt }1 \\
\beta_2-1 \text{ of wt }2\\
\beta_3-1 \text{ of wt }3\\
\beta_4 \text{ of wt }4\\
\vdots
\end{array}
$}}
\end{picture}
\end{center}

\end{example}

We denote by $Q_{(\Lambda,
A, B)}(\alpha; \beta)$ the number of linear orderings on $\P(\Lambda, A, B)$
(up to equivalence)
that extend the linear order on $\Lambda$. As $d \ge
d_{\min}(\Lambda, A, B)$ if and only if~(\ref{eqn:numberofshortedges}) is
non-negative, for $1 \le i \le l$, we have 
\begin{displaymath}
d_{\min}(\Lambda, A, B) = \max_{1 \le i \le l(\Lambda)}(
l(\Lambda) - i + 1 + \varkappa_i(\Lambda) + \wls(A)_{l(\Lambda) + 1 -
  i} + \wls(B)_{l(\Lambda)+1 - i}).
\end{displaymath}

\begin{example}
Recall from Example~\ref{ex:templatecomposition} that, to construct the
floor diagram $\D$ from the extended template $(\Lambda, A, B)$ (and
the two templates $\Gamma_1$ and $\Gamma_2$), we had to add
(cf.~(\ref{eqn:numberofshortedges})) precisely  
$d-5$, $d-9$ and $d-9$ short edges between the last four vertices of
$\D$, where $d$ is the degree of $\D$. The invariant $d_{\min}(\Lambda,
A, B,)$ measures for which $d$ this is possible. In this example, we
have $d_{\min}(\Lambda, A, B) = \max(5, 9, 9) = 9$.
\end{example}

For sequences $s, t_1, t_2, \dots$ with $s \ge \sum_i t_i$ (component-wise), we denote by
\smallskip
\begin{displaymath}
\binom{s}{t_1, t_2, \dots} \stackrel{\text{def}}{=} \frac{s!}{t_1 !  t_2 !   \cdots
(s - \sum_i t_i)!}
\end{displaymath}
\smallskip
the multinomial coefficient of sequences.

We obtain all $(\alpha, \beta)$-markings of the floor
diagram $\D$ that come from a compatible pair of sequences $(\{
\alpha^i \}, \{ \beta^i\})$ by independently ordering the
$\alpha$-vertices and the non-$\alpha$-vertices. 
The number of such markings is (via the correspondence
(\ref{eqn:relativedecomposition}))
\begin{equation}
\label{eqn:tailfactor}
\Big( \prod_{s=1}^m P_{\Gamma_s}(k_s) \Big)\cdot  \binom{\alpha}{a^T_1, a^T_2, \dots} \cdot Q_{(\Lambda,
A, B)}(\alpha; \beta),
\end{equation}
where $a^T_1, a^T_2, \dots$ are the {\it column} vectors of $A$. We conclude this section by
recasting relative Severi degrees in terms of templates and extended
templates.

\begin{proposition}
\label{prop:mainformula}
For any $\delta \ge 1$, the relative Severi degree $N_{\alpha,
  \beta}^\delta$ is given by
\smallskip
\begin{equation}
\label{eqn:mainformula}
\hspace{-8mm} 
\sum_{\tiny
\begin{array}{c}
(\Gamma_1, \dots, \Gamma_m), \\
 (\Lambda, A, B)
\end{array}
} \hspace{-3mm} \Big( \prod_{s =
  1}^m\mu(\Gamma_s) \hspace{-1mm}\sum_{k_1, \dots k_m} \prod_{s=1}^m
P_{\Gamma_s}(k_s) \Big) \hspace{-0.0mm} \cdot \hspace{-0.0mm} \Big( \mu(\Lambda) \prod_{i \ge 1} i^{\beta_i} \binom{\alpha}{a_1,
  a_2, \dots} Q_{(\Lambda,A, B)}(\alpha; \beta) \Big),
\hspace{-3mm}
\end{equation}
\smallskip
where the first sum is over all collections $(\Gamma_1, \dots,
\Gamma_m)$ of templates and all extended templates $(\Lambda, A, B)$
satisfying (\ref{eqn:inequalities2}), $d \ge d_{\min}(\Lambda, A, B)$ and
\smallskip
\begin{displaymath}
\sum_{i = 1}^m \delta(\Gamma_i) + \delta(\Lambda) + \delta(A) +
\delta(B) = \delta,
\end{displaymath}
\smallskip
and the second sum is over all positive integers $k_1, \dots,
k_m$ satisfying (\ref{eqn:inequalities}).
\end{proposition}

\begin{proof}
By Theorem \ref{thm:relcorrespondence}, the relative Severi degree is
given by
\begin{displaymath}
N^{\delta}_{\alpha, \beta} = \sum_{\D}
\mu_\beta(\D) \nu_{\alpha, \beta}(\D),
\end{displaymath}
the sum ranging over all floor diagrams $\D$ of degree $d = \sum_{i
  \ge 1} i (\alpha_i + \beta_i)$ and cogenus~$\delta$.
The result follows from $\mu_\beta(\D) = \prod_{i \ge 1} i^{\beta_i} \cdot  \big( \prod_{s = 1}^m
\mu(\Gamma_s) \big) \cdot \mu(\Lambda)$ and (\ref{eqn:tailfactor}).
\end{proof}

\section{Relative Severi Degrees  and Polynomiality}
\label{sec:proofs}

We now turn to the proofs of our main results by first showing a
number of technical lemmas.
For a graph $G$, we denote by $\#E(G)$ the number of
edges of~$G$. We write $||A||_1 = \sum_{i, j
  \ge 1} |a_{ij}|$ for the $1$-norm of a (possibly infinite) matrix $A = (a_{ij})$.

\begin{lemma}
\label{lem:polyoffinaltemplate}
For every extended template $(\Lambda, A, B)$, there is a polynomial $q_{(\Lambda,A, B)}$ in $\alpha_1, \alpha_2, \dots,\beta_1,
\beta_2, \dots$ of degree $\#E(\Lambda) + ||B||_1 + \delta(B)$ such that, for all
$\alpha$ and $\beta$ satisfying (\ref{eqn:inequalities2}),
the number $Q_{(\Lambda,A, B)}(\alpha; \beta)$ of linear orderings (up
to equivalence) of
the poset $\P(\Lambda, A, B)$ is given by
\begin{displaymath}
Q_{(\Lambda,A, B)}(\alpha; \beta) = \frac{(|\beta| -
  \delta(B) )!}{\beta !} \cdot q_{(\Lambda,
A, B)}(\alpha; \beta)
\end{displaymath}
provided
$\sum_{i \ge 1} i (\alpha_i + \beta_i) \ge d_{\min}(\Lambda, A, B)$.
\end{lemma}

\begin{proof}
We can choose a linear extension of the order on the vertices of
$\Lambda$ to the poset $\P(\Lambda, A, B)$ in two steps. First, we choose a linear order on the
vertices $0, \dots, l(\Lambda) \!+\!1$, the midpoint vertices of the edges
of $\Lambda$, and the midpoint vertices of the edges 
created in Step (\ref{itm:beta-edges}) in the definition of
$\P(\Lambda,A,B)$. In a second step, we choose an extension to a
linear order on all vertices. There are only finitely many choices in
the first step (in particular, they do not involve $\alpha$ and
$\beta$). Thus, for each choice in the first step, the number of
linear extensions in second step is of the desired form.

Let $r_i$ be the number of vertices between
$i-1$ and $i$ after the first extension, for $1 \le i \le l(\Lambda) +
1$, and let $\sigma_i$ be the number of equivalent such
linear orderings of the interval between $i - 1$ and $i$ ($\sigma_i$ is independent of the particular choice
of the linear order). To
insert the additional vertices (up to equivalence) between the vertices
$0$ and $l = l(\Lambda)$ we have
\begin{equation}
\label{eqn:numberoforderings}
\prod_{i = 1}^{l} \frac{1}{\sigma_i} \binom{d - l(\Lambda) + i - 1 - \varkappa_i(\Lambda) - \wls(A)_{l + 1 -
  i} - \wls(B)_{l+1 - i} + r_i}{r_i}
\end{equation}
many possibilities where again $d = \sum_{i \ge 1} i (\alpha_i + \beta_i)$. If $d \ge d_{\min}(\Lambda, A, B)$ then expression (\ref{eqn:numberoforderings}) is a polynomial in
$d$ of degree $\sum_{i=1}^l r_i$, and thus in $\alpha_1, \alpha_2,
\dots, \beta_1, \beta_2, \dots$. The number of (equivalent) orderings of the
vertices between $l$ and $l + 1$ is the multinomial coefficient
\begin{equation}
\label{eqn:numberoforderings2}
\binom{|\beta | - ||B||_1 +
  r_{l+1}}{\beta_1 -|b^T_1|, \beta_2 - |b^T_2|, \dots },
\end{equation}
where $|b^T_j|$ denotes the sum of the entries in the $j$th column of $B$.
As $||B||_1 \le \delta(B)$, expression
(\ref{eqn:numberoforderings2}) equals, for all $\beta_1, \beta_2, \dots \ge 0$,
\begin{equation}
\label{eqn:numberoforderings3}
\binom{|\beta|}{\beta_1, \beta_2, \dots} \frac{(|\beta| -
  \delta(B) )!}{|\beta|!} P(\beta) = \frac{(|\beta| -
  \delta(B) )!}{\beta !} P(\beta)
\end{equation}
for a polynomial $P$ in $\beta_1, \beta_2, \dots$ of degree
$r_{l+1} + \delta(B)$.
The product of
(\ref{eqn:numberoforderings}) and (\ref{eqn:numberoforderings3}) is
\begin{equation}
\label{eqn:product}
\frac{(|\beta| -
  \delta(B) )!}{\beta !} P'(\alpha; \beta)
\end{equation}
for a polynomial $P'$ in $\alpha_1, \alpha_2, \dots, \beta_1,
\beta_2, \dots$ of degree
$\#E(\Lambda) + ||B||_1 + \delta(B) $, provided $d \ge
d_{\min}(\Lambda, A, B)$, where we used that $\sum_{i = 1}^{l + 1} r_i =
\#E(\Lambda) + ||B||_1$. As (\ref{eqn:product}) equals
the number of linear extensions (up to equivalence) that can be obtained
by linearly ordering the vertices in all segments between $i - 1$ and
$i$, for $1 \le i \le l + 1$, the proof is complete.
\end{proof}

In Section~\ref{sec:relativenodepolys}, we defined, for an extended
template $(\Lambda, A, B)$, the invariant 
\begin{equation}
\label{eqn:dmin}
d_{\min} = \max_{1 \le i \le l(\Lambda)}( l(\Lambda) - i + 1 + \varkappa_i(\Lambda) + \wls(A)_{l(\Lambda)+1-i} + \wls(B)
_{l(\Lambda)+1-i}).
\end{equation}
It equals the minimal $d \ge 1$ so that $(\Lambda, A, B)$ can appear
in the relative decomposition of a floor diagram of degree $d$. 
Let $i_0$ be the smallest $i$ for
which the maximum in (\ref{eqn:dmin}) is attained.
Define the quantity $s(\Lambda, A, B)$ to be
the number of edges of $\Lambda$ from $i_0 - 1$ to $i_0$ (of any
weight). For example, if$(\Lambda, A, B)$ is the extended template
\begin{center}
\begin{picture}(270,74)(62,-35)\setlength{\unitlength}{3pt}\thicklines
\multiput(0,0)(10,0){4}{\circle{2}}
\eOee
 \put(16,2.3){\vector(1,0){1}}
 \put(16,-2.3){\vector(1,0){1}}

\put(18,3){\makebox(0,0){$2$}}
\put(15,-4){\makebox(0,0){$3$}}

 \qbezier(0.8,0.6)(6,5)(15,5)\qbezier(15,5)(24,5)(29.2,0.6)
 \put(16,5){\vector(1,0){1}}

\put(21,0){\line(1,0){8}}
 \put(24,2){\makebox(0,0){$2$}}
 \put(26,0){\vector(1,0){1}} 

\put(50,0){\footnotesize $A = 
\begin{bmatrix}
\, \, 1 \, \,& \, \, 0 \, \,& \, \,  0 \, \,&\cdots \\ 
1 & 0 & 0 & \cdots \\
0 & 0 & 0 & \cdots \\
\vdots & \vdots & \vdots & \ddots\\
\end{bmatrix}
\quad \quad
B = 
\begin{bmatrix}
\, \, 0 \, \,& \, \, 1 \, \,& \, \,  0 \, \,&\cdots \\
0 & 0 & 0 & \cdots \\
0 & 0 & 0 & \cdots \\
\vdots & \vdots & \vdots & \ddots\\
\end{bmatrix}
$
}
\end{picture}
\end{center}
then tracing through the definition yields $d_{\min} = \max(4,11,8) = 11$. The maximum is attained at $i_0 =
2$, and there are $s = 2$ edges in $\Lambda$ between vertices $1$
and $2$.

\begin{lemma}
\label{lem:smallevaluations}
For any extended
template $(\Lambda, A, B)$ and any $\alpha, \beta \ge 0$
(component-wise) with
\[ d_{\min}(\Lambda, A, B) - s(\Lambda, A, B) \le \sum_{i \ge 1} i
(\alpha_i + \beta_i) \le d_{\min}(\Lambda, A, B) - 1,\]
we have $q_{(\Lambda, A. B)}(\alpha; \beta)= 0$, where $q_{(\Lambda,
  A, B)}$ is the polynomial of
Lemma~\ref{lem:polyoffinaltemplate}.
\end{lemma}

\begin{proof}
Notice that $d_{\min} - l(\Lambda) + i_0 - 1= \varkappa_{i_0}(\Lambda) + \wls(A)_{l(\Lambda)+1-{i_0}} + \wls(B)
_{l(\Lambda)+1-{i_0}},$ where $d_{\min} = d_{\min}(\Lambda, A,
B)$. Therefore, the number of short edges added 
between $i_0 - 1$ and $i_0$ in Step~(3) of the definition of the
poset $\P(\Lambda, A, B)$ is $d - d_{\min}$, where as before $d =
\sum_{i \ge 1} i (\alpha_i + \beta_i)$. Recall that, up to the factor
$\tfrac{(|\beta| - \delta(B))!}{\beta!}$, the polynomial $q_{(\Lambda,
  A, B)}$ records the number of linear extensions of the poset
$\P(\Lambda, A, B)$ (up to equivalence). Every such extension is
obtained by, first, linearly ordering the $d - d_{\min}$ midpoints of
the short edges between $i_0 - 1$ and $i_0$ that were added in
Step~(3) together with the $s(\Lambda, A, B)$ midpoints of the edges
of $\Lambda$ between $i_0 - 1$ and $i_0$, before extending this to a
linear order on all the vertices of $\P(\Lambda, A, B)$. Let $r_{i_0}$
be the number of midpoint vertices between $i_0-1$ and $i_0$ after a particular
first extension.
Then $\binom{d - d_{\min} + r_{i_0}}{r_{i_0}}$
is a factor of the
polynomial counting the linear extensions on all vertices. As $r_{i_0}
\ge s$ for each choice of first extension,
$q_{(\Lambda,  A, B)}$ is divisible by the 
polynomial $(d - d_{\min} + 1) \cdots (d - d_{\min} + s(\Lambda, A,
B))$.
\end{proof}

The next lemma specifies the extended templates compatible with
a given degree.

\begin{lemma}
\label{lem:dmininequality}
For every extended template $(\Lambda, A, B)$ we have
\[
d_{\min}(\Lambda, A, B) - s(\Lambda, A, B)
\le \delta(\Lambda) + \delta(A) + \delta(B) + 1.
\]
\end{lemma}

\begin{proof}
We use the notation from above and write $l = l(\Lambda)$. Notice that
\[
d_{\min}(\Lambda, A, B) - l(\Lambda) + i_0 - 1=
\varkappa_{i_0}(\Lambda) + \wls(A)_{l+1-{i_0}} + \wls(B)_{l+1-{i_0}}.
\]
Therefore, it suffices to show 
\[
l(\Lambda) \le \delta(\Lambda) - \varkappa_{i_0}(\Lambda) + s(\Lambda,
A, B) + \delta(A) -
\wls(A)_{l+1-{i_0}} + \delta(B) - \wls(B)_{l+1-{i_0}} + i_0 .
\]
Let $\Lambda'$ be the graph obtained from $\Lambda$ by, firstly, removing all
edges $j \to k$ with either $k < i_0$ or $j \ge i_0$ and, secondly,
deleting all vertices $j$ for which there is no edge $i \to k$ in the
new graph with $i \le j \le k$. It is easy to
see that $l(\Lambda, A, B) - l(\Lambda', A, B) \le \delta(\Lambda) -
\delta(\Lambda')$. Thus, we can assume without loss of generality that
all edges $j \to k$ of $\Lambda$ satisfy $j < i_0 \le k$. Therefore,
as $\varkappa_{i_0}(\Lambda) = \sum_{\text{edges }e}
\wt(e)$, we have
\[
\delta(\Lambda) - \varkappa_{i_0} + s(\Lambda, A, B) =
\!\!\!\!\!\sum_{\text{edges }e} \!\!\!\! \Big( \wt(e)(\len(e)-1) -1 \Big) +
s = \!\!\! \!\!\!\!\!\! \sum_{e: \, \len(e) \ge 2} \! \!\!\!\!\!\!\!\!\!  \Big(
\wt(e)(\len(e)-1) -1 \Big),
\]
where, again, $\len(e)$ is the length $k-j$ of an edge $j \stackrel{e}{\to}
k$.  Thus, we need to show that
\begin{equation}
\label{eqn:generalcase}
\begin{split}
l(\Lambda) \le & \!\!\! \sum_{e: \, \len(e) \ge 2} \!\!\!\! \Big(
\wt(e)(\len(e)-1) -1 \Big) \\
& + \delta(A) - \wls(A)_{l+1-{i_0}} + \delta(B) -
\wls(B)_{l+1-{i_0}} + i_0.
\end{split}
\end{equation}
It is easy to see that the
matrix $A$ satisfies $\delta(A) \ge
\wls(A)_{i} + l(A) - 1$ for all $i \ge 1$, therefore, if $l(A) =
l(\Lambda)$, it suffices to show that
\begin{equation}
\label{eqn:caseA}
l(A) \le \!\!\! \sum_{e: \, \len(e) \ge 2} \!\!\!\! \Big(
\wt(e)(\len(e)-1) -1 \Big) + l(A) - 1 + \delta(B) -
\wls(B)_{l+1-{i_0}} + i_0.
\end{equation}
But (\ref{eqn:caseA}) is clear as all summands in the sum
over the edges of $\Lambda$ are non-negative and
$\delta(B) \ge \wls(B)_{l+1-{i_0}}$. The same argument also settles
the case $l(B) = l(\Lambda)$.

Otherwise, we can assume that $l(\Lambda) > l(A)
\ge l(B)$ and that there exists an
edge $0 \to i$ of $\Lambda$ with $l(\Lambda) - l(A) \le
i - 1$. Assume, additionally, that $i_0 \le l(\Lambda) - l(A)$ which implies
$\wls(A)_{l+1-i_0} = 0$.Then we have $l(\Lambda) \le i - 1 + l(A)$,
and, again using $\delta(B) \ge \wls(B)_{l+1-{i_0}}$, it suffices to
show that
\begin{displaymath}
i - 1 + l(A) \le \sum_{e: \, \len(e) \ge 2} \Big(
\wt(e)(\len(e)-1) -1 \Big) + \delta(A) + i_0.
\end{displaymath}
But this inequality is clear as
\begin{displaymath}
 \sum_{e: \, \len(e) \ge 2} \Big(
\wt(e)(\len(e)-1) -1 \Big) \ge i - 2,
\end{displaymath}
together with $l(A) \le \delta(A)$ and $i_0 \ge 1$.

Finally, it remains to show that $l(\Lambda) > l(A) \ge l(B)$ and $i_0
\ge l(\Lambda) - l(A) + 1$ imply (\ref{eqn:generalcase}). Using $\sum \Big(
\wt(e)(\len(e)-1) -1 \Big) \ge 0$ and 
$\delta(B) \ge \wls(B)_{l+1-{i_0}}$, it suffices to show
\begin{equation}
\label{eqn:lastcase}
l(\Lambda) \le \delta(A) - \wls(A)_{l+1-i_0} +i_0.
\end{equation}
We have (by definition of $\delta(A)$ and
$\wls(A)_{l + 1 - i_0}$) that
\begin{equation}
\label{eqn:stupidinequality}
\delta(A) - \wls(A)_{l+1-i_0} +i_0 = \sum (i-1) j a_{ij} + \sum i j
a_{ij} + i_0,
\end{equation}
where the first sum runs over $i \ge l + 1 - i_0$, $j \ge 1$ and
the second sum runs over $1 \le i < l + 1 - i_0$, $j \ge 1$. As $i_0
\ge l(\Lambda) - l(A) + 1$ there exists a non-zero entry $a_{i'j'}$ of $A$
with $i' = l(A) \ge l + 1 - i_0$. Therefore, the index set of the first sum of~(\ref{eqn:stupidinequality}) is non-empty and the right-hand side of
(\ref{eqn:stupidinequality}) is greater or equal to $i' - 1 + i_0 = l(\Lambda)$ as $i_0
\ge l(\Lambda) - l(A) + 1$. This show (\ref{eqn:lastcase}) and completes
the proof.
\end{proof}

Before we turn to the proof of the main theorem of this paper, we
introduce a last numerical invariant $s(\Gamma)$ associated with each template
$\Gamma$. The definition of $s(\Gamma)$ parallels that of $s(\Lambda,
A, B)$ for extended templates $(\Lambda, A, B)$. This invariant is
necessary to establish the polynomiality threshold of
Theorem~\ref{thm:relativenodepoly}.

Recall from Section~\ref{sec:relativenodepolys} that, for a template
$\Gamma$, we defined
\begin{displaymath}
k_{\min}(\Gamma) = \max_{1 \le j \le l(\Gamma)} (\varkappa_j(\Gamma) - i + 1),
\end{displaymath} 
where $\varkappa_j(\Gamma)$ is the total weight of all edges $i \to k$
of $\Gamma$ with $i < j \le k$. Let $j_0(\Gamma)$ be the smallest $j$
for which the maximum is attained and define $s(\Gamma)$ to be the
number of edges of $\Gamma$ from $j_0 - 1$ to $j_0$. See
Figure~\ref{fig:templates} for examples. Then one can show the
following; the proof is along
similar lines as for our Lemma~\ref{lem:dmininequality}.

\begin{lemma}[{\cite[Lemma~4.3]{FB}}]
\label{lem:s_correction}
Each template $\Gamma$ satisfies
\begin{displaymath}
k_{\min}(\Gamma) + l(\Gamma) - s(\Gamma) \le \delta(\Gamma) + 1.
\end{displaymath}
\end{lemma}

{\it Proof of Theorem~\ref{thm:relativenodepoly} }
We first show that (\ref{eqn:relativenodepoly}) holds of all $\alpha$,
$\beta$ with $d \ge \delta + 1$, where we again write $d = \sum_{i \ge 1} i (\alpha_ i + \beta_i)$. This implies
(\ref{eqn:relativenodepoly}) for all $\alpha$ and $\beta$, for which at least one of $\alpha_1,
\alpha_2, \dots,$ $\beta_2,
\beta_3, \dots$ is non-zero (note that $\beta_1$ is omitted):
in that case $|\beta|$~\!\!~$\ge$~\!\!~$ \delta$ implies $d  \ge \delta + 1$.

Notice that we can remove condition~(\ref{eqn:inequalities2}) from formula (\ref{eqn:mainformula}) of
Proposition~\ref{prop:mainformula} and still obtain correct
relative Severi degrees as $\binom{\alpha}{a^T_1,
  a^T_2, \dots} Q_{(\Lambda, A, B)}(\alpha; \beta) = 0$ whenever
(\ref{eqn:inequalities2}) is violated.
The first factor of (\ref{eqn:mainformula}) equals
\begin{equation}
\label{eqn:firstfactor}
\sum_{k_m = k_{\min} (\Gamma_m)}^{d-l(\Lambda)} \mu(\Gamma_m) P_{\Gamma_m}(k_m)
\sum_{k_{m-1} = k_{\min}(\Gamma_{m-1})}^{k_m-l(\Gamma_{m-1})} \cdots
\sum_{k_1 = k_{\min}(\Gamma_1)}^{k_2-l(\Gamma_1)} \mu(\Gamma_1) P_{\Gamma_1}(k_1)
\end{equation}
and is, therefore, an iterated ``discrete integral'' of
polynomials.
Expression (\ref{eqn:firstfactor}) is polynomial in $d$, provided $d$
is large enough (as discrete integration preserves polynomiality, see
e.g. Faulhaber's formula~\cite[Lemma~3.5]{FB}). Furthermore, as the polynomials $P_{\Gamma_i}(k_i)$
have degrees $\#E(\Gamma_i)$ and each ``discrete integration''
increases the degree by $1$, the polynomial~(\ref{eqn:firstfactor})
(if $d$ is large enough) is of degree $\sum_{i = 1}^m
\#E(\Gamma_i) + m$.

We claim that~(\ref{eqn:firstfactor}) is a polynomial in $d$ provided
that $d - l(\Lambda) \ge \sum_{i=1}^m \delta(\Gamma_i) + 1$. (Readers
interested solely in polynomiality of relative Severi degrees for
large enough $\alpha$ and $\beta$ may skip this
paragraph. Also, computational evidence suggests that this bound,
without the ``$+1$" is sharp in general). Indeed, by
\cite[Lemma~3.6]{FB} and repeated application of \cite[Lemma~4.1]{FB}
and \cite[Lemma~4.2]{FB}, it suffices to show that $d - l(\Lambda) \ge
\sum_{i=1}^m \delta(\Gamma_i) + 1$ simultaneously implies
\begin{equation}
\label{eq:manyinequal}
\begin{split}
d \ge & \, l(\Gamma_m) + k_{\min}(\Gamma_m) - s(\Gamma_m)-1, \\
d \ge & \, l(\Gamma_m) + l(\Gamma_{m-1}) +
      k_{\min}(\Gamma_{m-1}) - s(\Gamma_{m-1})- 2, \\
& \quad \quad \quad \quad \quad \quad \vdots \\
d \ge & \, l(\Gamma_m) + l(\Gamma_{m-1}) + \cdots +
l(\Gamma_1) + k_{\min}(\Gamma_1) - s(\Gamma_1)- m,
\end{split}
\end{equation}
for all collections of templates $(\Gamma_1, \dots, \Gamma_m)$ with
$\sum_{i =1}^m \delta(\Gamma_i) = \delta$.
The first inequality follows directly from
Lemma~\ref{lem:s_correction}. For the other inequalities, notice
that $l(\Gamma_i) - 1 \le \delta(\Gamma_i)$ for all templates
$\Gamma_i$ in the collection. Fix $1 \le i_0 \le m$, then 
Lemma~\ref{lem:s_correction} applied to $\Gamma_{i_0}$ says
\begin{displaymath}
l(\Gamma_{i_0}) + k_{\min}(\Gamma_{i_0}) - s(\Gamma_{i_0})  \le
\delta(\Gamma_{i_0}) + 1.
\end{displaymath}
Thus the right-hand-side of the $(m - i_0 + 1)$th inequality of
(\ref{eq:manyinequal}) is less or equal than
\begin{displaymath}
\sum_{i=i_0}^m \delta(\Gamma_i) + 1 \le d - l(\Lambda) \le d.
\end{displaymath} 

Furthermore, we have
$l(\Lambda)  \le \delta(\Lambda) + \delta(A) + \delta(B)$.
(By the ``connectedness'' property of extended templates, each vertex $1 \le j \le
l(\Lambda) - \max(l(A), l(B))$ of $\Lambda$ is passed by an edge of $\Lambda$.
Therefore, we have $\delta(\Lambda)) \ge l(\Lambda) - \max(l(A),
l(B))$. Notice that $l(A) \le \delta(A)$ and $l(B) \le \delta(B)$.)
Therefore, 
the first factor of (\ref{eqn:mainformula}) is a polynomial in $d$ if $d \ge \delta + 1
 = \big( \sum_i \delta(\Gamma_i) \big) + 1 + \delta(\Lambda) + \delta(A) +
 \delta(B)$ (again computation evidence suggests that this bound,
 without the ``$+1$", is sharp).

For column vectors $a^T_1, a^T_2, \dots$ of a matrix $A$, the
multinomial coefficient $\binom{\alpha}{a^T_1, a^T_2, \dots}$ is a
polynomial of degree $||A||_1$ in $\alpha_1, 
\alpha_2, \dots$, if $\alpha_1,
\alpha_2, \dots \ge 0$. For an extended template $(\Lambda, A,
B)$, define 
\begin{displaymath}
R_{(\Lambda, A, B)}(\alpha; \beta) \stackrel{\text{def}}{=}
\binom{\alpha}{a^T_1, a^T_2, \dots} \cdot \frac{(|\beta| -
  \delta(B))!}{(|\beta| - \delta)!} \cdot q_{(\Lambda, A, B)}(\alpha; \beta),
\end{displaymath}
where $q_{(\Lambda, A, B)}(\alpha; \beta)$ is the polynomial of
Lemma~\ref{lem:polyoffinaltemplate}. By
Lemma~\ref{lem:polyoffinaltemplate} and as $\delta(B) \le
\delta$, $R_{(\Lambda, A, B)}(\alpha; \beta)$ is a polynomial in
$\alpha$ and $\beta$ of degree $\#E(\Lambda) + ||A||_1 +||B||_1 + \delta$, provided $d \ge
d_{\min}(\Lambda, A, B)$. The second factor of
(\ref{eqn:mainformula}) then equals
\begin{equation}
\label{eqn:2ndfactorcontribution}
\prod_{i \ge 1} i^{\beta_i} \cdot
\frac{(|\beta| - \delta)!}{\beta !} \cdot 
R_{(\Lambda, A, B)}(\alpha; \beta).
\end{equation}
By Lemma~\ref{lem:smallevaluations}, the second factor of
(\ref{eqn:mainformula}) equals expression (\ref{eqn:2ndfactorcontribution}) for
all $\alpha, \beta$ with $d \ge d_{\min}(\Lambda, A,
B) - s(\Lambda, A, B)$. Thus, using  Lemma~\ref{lem:dmininequality}, if 
\[
d \ge \delta + 1 \ge \delta(\Lambda) + \delta(A) + \delta(B) + 1 \ge d_{\min}(\Lambda, A,
B) - s(\Lambda, A, B)
\]
the second
factor in (\ref{eqn:mainformula}) is $\prod_{i \ge 1} i^{\beta_i}
\cdot \tfrac{(|\beta| -
  \delta)!}{\beta!}$ times a polynomial in $\alpha_1,
\alpha_2, \dots, \beta_1, \beta_2, \dots$ of degree $\#E(\Lambda) +
||A||_1 + ||B||_1 + \delta $. Hence
(\ref{eqn:relativenodepoly}) holds if $|\beta| \ge \delta$ and at
least one $\beta_i$, for $i \ge 2$, or one $\alpha_i$, for $i
\ge 1$, is non-zero. Notice that each summand of
(\ref{eqn:mainformula}) contributes a polynomial of degree
\begin{equation}
\label{eqn:totaldegree}
\sum_{i=1}^m \#E(\Gamma_i) + m + \#E(\Lambda) + ||A||_1 + ||B||_1 + \delta 
\end{equation}
to the relative node polynomial $N_{\delta}(\alpha; \beta)$. It is not
hard to see that expression (\ref{eqn:totaldegree}) is at most $3\delta$,
and that equality is attained by letting $\Gamma_1, \dots,
\Gamma_\delta$ be the unique template on three vertices with cogenus
$1$ (see Figure~\ref{fig:templates}) and $(\Lambda, A, B)$ be the
unique extended template of cogenus $0$ (see Figure~\ref{fig:extendedtemplates}).

If $\alpha = 0$ and $\beta = (d, 0, \dots)$, then $N^\delta_{\alpha,
  \beta}$ equals the (non-relative) Severi degree $N^{d, \delta}$, which, in
turn, is given by the (non-relative) node polynomial $N^\nr_\delta(d)$
provided $d \ge \delta$ (see \cite[Theorem 1.3]{FB}). Therefore, we
have $N_\delta(0; d) = N^\nr_\delta(d) \cdot d (d-1) \cdots (d-\delta
+ 1)$ as polynomials in $d$. Applying \cite[Theorem 1.3]{FB} again
finishes the proof.

\smallskip

\begin{remark}
\label{rmk:algorithm}
Expression (\ref{eqn:mainformula}) gives, in principle, an algorithm to
compute the relative node polynomial $N_\delta(\alpha; \beta)$, for
any $\delta \ge 1$. In \cite[Section 3]{FB} we
explain how to generate all templates of a given cogenus, and how to
compute the first factor in (\ref{eqn:mainformula}). The generation of
all extended templates of a given cogenus from the templates is
straightforward, as is the computation of the second factor in (\ref{eqn:mainformula}).
\end{remark}

\begin{remark}
The proof of Theorem~\ref{thm:relativenodepoly} simplifies if we relax
the polynomiality threshold. More specifically, without considering the quantity $s(\Lambda, A, B)$ and the rather technical
Lemmas~\ref{lem:smallevaluations} and~\ref{lem:dmininequality}, the
argument still implies (\ref{eqn:relativenodepoly}) provided $|\beta|$~\!\!~$\ge$~\!\!~$ 2 \delta$ (instead of $|\beta| \ge \delta$).
\end{remark}

The conclusion from the proof of Theorem
\ref{thm:relativenodepoly} is two-fold.

\medskip

{\it Proof of Proposition~\ref{prop:variables} }
Every extended template $(\Lambda, A, B)$ considered in
(\ref{eqn:mainformula}) satisfies $\delta(A) \le \delta$ and $\delta(B)
\le \delta$. Therefore, all rows $i > \delta$ in $A$ or $B$ are zero.

\medskip

{\it Proof of Theorem~\ref{thm:stability} }
By the proof of Lemma~\ref{lem:polyoffinaltemplate} we have, for every
extended template $(\Lambda, A, B)$,
\[
R_{(\Lambda, A, B)}(\alpha, 0; \beta) = R_{(\Lambda, A, B)}(\alpha;
\beta) \quad \quad R_{(\Lambda, A, B)}(\alpha; \beta, 0) = R_{(\Lambda, A, B)}(\alpha;
\beta).
\]
Hence, by the proof of Theorem~\ref{thm:relativenodepoly}, the result follows.

\medskip

Now it is also easy to prove Theorem~\ref{thm:newrelativenodepolys}.

\medskip

{\it Proof of Theorem~\ref{thm:newrelativenodepolys} }
Proposition~\ref{prop:mainformula} gives a combinatorial description
of relative Severi degrees. The proof of
Lemma~\ref{lem:polyoffinaltemplate} provides a method to calculate the polynomial
$Q_{(\Lambda, A, B)}(\alpha; \beta)$.
All terms of expression~(\ref{eqn:mainformula}) are explicit or can be
evaluated using the techniques of \cite[Section 3]{FB}. This reduces
the calculation to a (non-trivial) computer calculation.

\section{Coefficients of Relative Node Polynomials}
\label{sec:relativecoefficients}

We now turn toward the computation of the coefficients of the relative node
polynomial $N_\delta(\alpha; \beta)$ of large degree for any~
$\delta$. By Theorem~\ref{thm:relativenodepoly}, the polynomial $N_\delta(\alpha,
\beta)$ is of degree $3 \delta$. In the following we propose a method
to compute all terms of $N_\delta(\alpha; \beta)$ of degree $\ge 3
\delta - t$, for any given $t \ge 0$. This method was used (with $t =
2$) to compute
the terms in Theorem~\ref{thm:coefficients}.

The main idea of the algorithm is that, even for general $\delta$,
only a small number of summands of (\ref{eqn:mainformula}) contribute to the terms of
$N_\delta(\alpha; \beta)$ of large degree. A summand of (\ref{eqn:mainformula}) is indexed by a collection of
templates $ \tilde{\Gamma} = \{ \Gamma_s\}$ and an extended template
$(\Lambda, A, B)$. To determine whether this summand actually contributes to
$N_\delta(\alpha; \beta)$ we define the \emph{(degree) defects}
\begin{itemize}
\item of the collection of templates $\tilde{\Gamma}$ by
$
\defect(\tilde{\Gamma}) \stackrel{\text{def}}{=} \big( \sum_{s=1}^m
\delta(\Gamma_i) \big) - m,
$
\item of the extended template $(\Lambda, A, B)$ by
\begin{displaymath}
\defect(\Lambda, A, B) \stackrel{\text{def}}{=} \delta(\Lambda) + 2\delta(A) + 2\delta(B)
 -||A||_1 - ||B||_1.
\end{displaymath}
\end{itemize}

The following lemma restricts the indexing set of (\ref{eqn:mainformula})
to the relevant terms, if only the leading terms of
$N_\delta(\alpha; \beta)$ are of interest. 

\begin{lemma}
The summand of (\ref{eqn:mainformula}) indexed by $\tilde{\Gamma}$ and
$(\Lambda, A, B)$ is of the form
\[
1^{\beta_1} 2^{\beta_2} \cdots \frac{(|\beta| - \delta)!}{\beta !} \cdot P(\alpha; \beta),
\]
where $P(\alpha; \beta)$ is a polynomial in $\alpha_1, \alpha_2,
\dots, \beta_1, \beta_2, \dots$ of degree $\le 3 \delta - \defect(\tilde{\Gamma}) -
\defect(\Lambda, A, B)$.
\end{lemma}

\begin{proof}
By \cite[Lemma 5.2]{FB}, the first factor of
(\ref{eqn:mainformula}) is of degree at most
\begin{displaymath}
2 \cdot \sum_{s=1}^m
\delta(\Gamma_s) - \sum_{s=1}^m (\delta(\Gamma_s) - 1) = \sum_{s =
  1}^m \delta(\Gamma_s) + m.
\end{displaymath}
The multinomial coefficient
$\binom{\alpha}{a^T_1, a^T_2, \dots}$ is a polynomial in $\alpha$ of
degree $||A||_1$ if $a^T_j$ are the $j$th column vector of the matrix $A$.
Recall, from the proof of Theorem~\ref{thm:relativenodepoly}, that the second factor of
(\ref{eqn:mainformula}) is
\begin{displaymath}
\prod_{i \ge 1} i^{\beta_i} \frac{(|\beta| - \delta)!}{\beta !} \mbox{ times a polynomial
  in }\alpha, \beta \text{ of
  degree } \#E(\Lambda) +||A||_1 + ||B||_1 + \delta.
\end{displaymath}
Therefore, the contribution of this summand is of degree at most
\begin{displaymath}
\begin{split}
&\sum_{s=1}^m \delta(\Gamma_s) + m + \#E(\Lambda) + ||A||_1+
||B||_1 + \delta \\
 = \, & 3\delta - 2 \sum_{s=1}^m
\delta(\Gamma_s) - 2 \delta(\Lambda) - 2 \delta(A) - 2 \delta(B) +
\#E(\Lambda) \\
= \, & 3\delta -\defect(\tilde{\Gamma}) - \defect(\Lambda, A, B) -
\delta(\Lambda) + \# E(\Lambda).
\end{split}
\end{displaymath}
The result follows as $\delta(\Lambda) \ge \# E(\Lambda)$.
\end{proof}

Therefore, to compute  the coefficients of degree $\ge 3\delta
- t$ of $N_\delta(\alpha; \beta)$ for some $t \ge 0$, it suffices to
consider only summands of  (\ref{eqn:mainformula}) with
$\defect(\tilde{\Gamma}) \le t$ and $\defect(\Lambda, A, B) \le t$.

One can proceed as follows. First, we can compute, for some formal
variable $\tilde{\delta}$, the terms of degree $\ge 2
\tilde{\delta} - t$ of the first factor of (\ref{eqn:mainformula}) to
$N_{\tilde{\delta}}(\alpha; \beta)$, that is the terms of degree $\ge 2
\tilde{\delta} - t$ of
\begin{equation}
\label{eqn:Rpoly}
R_{\tilde{\delta}}(d) \stackrel{\text{def}}{=} \sum \prod_{i = 1}^m
\mu(\Gamma_i) \sum_{k_m = k_{\min}(\Gamma_m)} ^{d - l(\Gamma_m)}
P_{\Gamma_m}(k_m) \cdots \sum_{k_1 = k_{\min}(\Gamma_1)}^{k_2 -
  l(\Gamma_1)} P_{\Gamma_1}(k_1),
\end{equation}
where the first sum is over all collections of templates
$\tilde{\Gamma} = (\Gamma_1, \dots, \Gamma_m)$ with
$\delta(\tilde{\Gamma}) = \tilde{\delta}$.
(Notice that (\ref{eqn:Rpoly})
is expression \cite[(5.13)]{FM} without the ``$\eps$-correction'' in the sum
indexed by~$k_m$.) The leading terms of $R_{\tilde{\delta}}(d)$ can be
computed with a slight modification of \cite[Algorithm 2]{FB} (by
replacing, in the notation of \cite{FB}, $C^{\text{end}}$ by $C$ and
$M^{\text{end}}$ by $M$). 
The algorithm 
relies on the polynomiality of solutions of certain polynomial
difference equations, which has been verified for $t \le 7$, see
\cite[Section 5]{FB} for more details.
With a Maple implementation of this algorithm one obtains
(with $t = 5$)
\medskip
\begin{displaymath}
\small
\begin{split}
R_{\tilde{\delta}}(d) &= \frac{3^{\tilde{\delta}}}{{\tilde{\delta}} !} \Big[ d^{2 {\tilde{\delta}}} - \frac{8
  {\tilde{\delta}}}{3} d^{2 {\tilde{\delta}} -1} + \frac{{\tilde{\delta}} (11 {\tilde{\delta}}+1) 
  }{3^2} d^{2 {\tilde{\delta}} - 2} + \frac{{\tilde{\delta}} ({\tilde{\delta}} - 1) (496 {\tilde{\delta}}
  - 245) }{6\cdot 3^3} d^{2 {\tilde{\delta}} - 3} \Big. \\
& - \frac{{\tilde{\delta}} ({\tilde{\delta}} - 1) (1685
  {\tilde{\delta}}^2 -2773 {\tilde{\delta}} + 1398) }{6 \cdot 3^4}
d^{2 {\tilde{\delta}} - 4} + \\
& \Big. - \frac{{\tilde{\delta}} ({\tilde{\delta}} - 1)({\tilde{\delta}} - 2) (7352 {\tilde{\delta}}^2 + 11611 {\tilde{\delta}} - 25221)
 }{30 \cdot 3^5} d^{2 {\tilde{\delta}} - 5} +
\cdots \Big].
\end{split}
\end{displaymath}

\medskip

Finally, to compute the coefficients of degree $\ge 3 \delta - t$, it
remains to compute all extended templates $(\Lambda, A, B)$ with
$\defect(\Lambda, A, B) \le t$ and collect the terms of degree $\ge 3
\delta - t$ of the polynomial
\begin{equation}
\label{eqn:collectterms}
R_{\tilde{\delta}}(d - l(\Lambda)) \cdot \mu(\Lambda)
\binom{\alpha}{a^T_1, a^T_2, \dots} \prod_{i = \delta(B)}^{\delta-1} (|\beta|
- i) \cdot
 q_{(\Lambda, A, B) } (\alpha; \beta),
\end{equation}
where, as before, $a^T_1, a^T_2, \dots$ denote the column vectors of the
matrix $A$, $ q_{(\Lambda, A, B) }(\alpha; \beta)$ is the polynomial
of Lemma~\ref{lem:polyoffinaltemplate}, and $\tilde{\delta} = \delta -
\delta(\Lambda, A, B)$. Notice that, for an indeterminant $x$ and 
integers $c \ge 0$ and $\delta \ge 1$, we have the expansion
\begin{equation}
\label{eqn:stirling}
\prod_{i = c}^{\delta
- 1} (x - i) = \sum_{t = 0}^{\delta -c} s(\delta - c, \delta - c - t) (x -
c)^{\delta - c - t},
\end{equation}
where $s(n,m)$ is the \emph{Stirling number
  of the first kind} \cite[Section 1.3]{St} for integers $n, m \ge
0$. Furthermore, with $\delta' = \delta - c$ the coefficients
$s(\delta', \delta' - t)$ of the right-hand-side of (\ref{eqn:stirling})
equal $\delta' (\delta' - 1) \cdots (\delta' - t) \cdot S_t(\delta')$, where $S_t$ is
the $t$th \emph{Stirling polynomial} \cite[(6.45)]{ABC}, for $t \ge 0$, and
thus are
polynomial in $\delta'$. Therefore, we can compute the leading
terms of the product in (\ref{eqn:collectterms}) by collecting the
leading terms in the sum expansion above.

\medskip

{\it Proof of
  Proposition~\ref{prop:coefficientpolynomiality} }
Using \cite[Algorithm 2]{FB} we can compute the terms of the polynomial $R_{\tilde{\Gamma}}(d)$ of degree $\ge
2\tilde{\delta} - 7$ (see
\cite[Section 5]{FB}) and observe that all coefficients are polynomial
in $\tilde{\delta}$. By the previous
paragraph, the coefficients of the expansion of the sum of
(\ref{eqn:collectterms}) are polynomial in $\delta$. This completes
the proof.

\smallskip



\medskip

{\it Proof of Theorem~\ref{thm:coefficients} }
The method described above is a direct implementation of 
formula~(\ref{eqn:mainformula}), which equals the relative Severi
degree by the proof of Theorem~\ref{thm:relativenodepoly}.

\smallskip

\begin{remark}
\label{rmk:morecoefficients}
It is straightforward to compute the coefficients of
$N_\delta(\alpha; \beta)$ of degree $\ge 3 \delta - 7$
(and thereby to extend Theorem~\ref{thm:coefficients}).
In particular, one can see that the terms $\alpha_2$ and $\beta_2$ (by
themselves) appear in $N_\delta(\alpha; \beta)$ in degree $3 \delta -3$.
Algorithm~3 of \cite{FB} computes the coefficients of the polynomials
$R_{\tilde{\delta}}(d)$ of degree $\ge 2 \tilde{\delta} - 7$, and thus the
desired terms can be collected from (\ref{eqn:collectterms}). We
expect this method to compute the leading terms of $N_\delta(\alpha,
\beta)$ of degree $\ge 3 \delta - t$ for arbitrary $t \ge 0$ (see \cite[Section 5]{FB}, especially
Conjecture~5.5).
\end{remark}

\appendix

\section{The first three Relative Node Polynomials}
\label{app:relativenodepolys}

Below we list the relative node polynomials $N_\delta(\alpha; \beta)$
for $\delta \le 3$. For $\delta \le 6$, the polynomials
$N_\delta(\alpha; \beta)$ are as
provided in the
ancillary files of this paper.
All polynomials were obtained by a Maple implementation of
the formula (\ref{eqn:mainformula}). See Remark~\ref{rmk:algorithm}
for more details. For $\delta \le 1$ this agrees with \cite[Corollary
4.5, 4.6]{FM}. As before, we write $d = \sum_{i \ge 1} i (\alpha_i + \beta_i)$. By
Theorem~\ref{thm:relativenodepoly} the relative Severi degrees
$N^\delta_{\alpha, \beta}$ are given by $N^\delta_{\alpha, \beta} =
1^{\beta_1} 2^{\beta_2} \cdots \tfrac{(|\beta| - \delta)!}{\beta !} N_\delta(\alpha, \beta)$
provided $|\beta| \ge \delta$.

The polynomials $N_4$, $N_5$ and $N_6$ have $599$, $1625$ and $3980$ terms,
respectively. The first few of their leading terms can be determined from Theorem~\ref{thm:coefficients}.

\begin{equation*}
\tiny
\begin{split}
N_0(\alpha, \beta) &= 1, \\
N_1(\alpha, \beta)
&=3 d^2 |\beta|-8 d |\beta|+d \beta_1+|\beta| \alpha_1+|\beta| \beta_1+4 |\beta|-\beta_1,\\
N_2(\alpha, \beta)
&=\tfrac{9}{2} d^4 |\beta|^2-\tfrac{9}{2} d^4 |\beta|-24 d^3
|\beta|^2+3 d^3 |\beta| \beta_1+3 d^2 |\beta|^2 \alpha_1+3 d^2
|\beta|^2 \beta_1+24 d^3 |\beta|-3 d^3 \beta_1+23 d^2 |\beta|^2 \\
&-3 d^2 |\beta| \alpha_1-14 d^2 |\beta| \beta_1+\tfrac{1}{2} d^2
\beta_1^2-8 d |\beta|^2 \alpha_1-8 d |\beta|^2 \beta_1+d |\beta|
\alpha_1 \beta_1+d |\beta| \beta_1^2+\tfrac{1}{2} |\beta|^2
\alpha_1^2  +|\beta|^2 \alpha_1 \beta_1 \\
& +\tfrac{1}{2} |\beta|^2 \beta_1^2-23 d^2
|\beta|+\tfrac{21}{2} d^2 \beta_1+\tfrac{3}{2} d |\beta|^2+8 d |\beta| \alpha_1+11 d |\beta|
\beta_1+d |\beta| \beta_2-d \alpha_1 \beta_1-\tfrac{5}{2} d
\beta_1^2-\tfrac{1}{2} |\beta|^2 \alpha_1 \\
& +|\beta|^2
\alpha_2-\tfrac{1}{2} |\beta|^2 \beta_1 +|\beta|^2
\beta_2-\tfrac{1}{2} |\beta| \alpha_1^2-3 |\beta| \alpha_1
\beta_1-\tfrac{5}{2} |\beta| \beta_1^2-\tfrac{83}{2} d
|\beta|-\tfrac{3}{2} d \beta_1-d \beta_2-48 |\beta|^2+\tfrac{1}{2}
|\beta| \alpha_1 \\
& -|\beta| \alpha_2+\tfrac{29}{2} |\beta| \beta_1-3 |\beta| \beta_2+2 \alpha_1 \beta_1+3 \beta_1^2+48 |\beta|-15 \beta_1+2 \beta_2,\\
N_3(\alpha, \beta)
&= \tfrac{9}{2} d^6 |\beta|^3-\tfrac{27}{2} d^6 |\beta|^2-36 d^5
|\beta|^3+\tfrac{9}{2} d^5 |\beta|^2 \beta_1+\tfrac{9}{2} d^4
|\beta|^3 \alpha_1+\tfrac{9}{2} d^4 |\beta|^3 \beta_1+9 d^6
|\beta|+108 d^5 |\beta|^2-\tfrac{27}{2} d^5 |\beta| \beta_1 \\
& +51 d^4 |\beta|^3-\tfrac{27}{2} d^4 |\beta|^2 \alpha_1-42 d^4
|\beta|^2 \beta_1+\tfrac{3}{2} d^4 |\beta| \beta_1^2-24 d^3 |\beta|^3
\alpha_1-24 d^3 |\beta|^3 \beta_1+3 d^3 |\beta|^2 \alpha_1 \beta_1+3
d^3 |\beta|^2 \beta_1^2 \\
& +\tfrac{3}{2} d^2 |\beta|^3 \alpha_1^2+3 d^2
|\beta|^3 \alpha_1 \beta_1 +\tfrac{3}{2} d^2 |\beta|^3 \beta_1^2-72
d^5 |\beta|+9 d^5 \beta_1-153 d^4 |\beta|^2+9 d^4 |\beta| \alpha_1+93
d^4 |\beta| \beta_1-3 d^4 \beta_1^2 \\
& +\tfrac{1243}{6} d^3 |\beta|^3+72 d^3 |\beta|^2 \alpha_1+92 d^3
|\beta|^2 \beta_1+3 d^3 |\beta|^2 \beta_2-9 d^3 |\beta| \alpha_1
\beta_1-\tfrac{35}{2} d^3 |\beta| \beta_1^2+\tfrac{1}{6} d^3
\beta_1^3+\tfrac{19}{2} d^2 |\beta|^3 \alpha_1 \\
& +3 d^2 |\beta|^3 \alpha_2+\tfrac{19}{2} d^2 |\beta|^3 \beta_1 +3 d^2
|\beta|^3 \beta_2-\tfrac{9}{2} d^2 |\beta|^2 \alpha_1^2-23 d^2
|\beta|^2 \alpha_1 \beta_1-\tfrac{37}{2} d^2 |\beta|^2
\beta_1^2+\tfrac{1}{2} d^2 |\beta| \alpha_1 \beta_1^2 \\
& +\tfrac{1}{2} d^2 |\beta| \beta_1^3-4 d |\beta|^3 \alpha_1^2-8 d
|\beta|^3 \alpha_1 \beta_1-4 d |\beta|^3 \beta_1^2+\tfrac{1}{2} d
|\beta|^2 \alpha_1^2 \beta_1+d |\beta|^2 \alpha_1
\beta_1^2+\tfrac{1}{2} d |\beta|^2 \beta_1^3+\tfrac{1}{6} |\beta|^3
\alpha_1^3 \\
& +\tfrac{1}{2} |\beta|^3 \alpha_1^2 \beta_1+\tfrac{1}{2} |\beta|^3
\alpha_1 \beta_1^2+\tfrac{1}{6} |\beta|^3 \beta_1^3+102 d^4 |\beta|-54
d^4 \beta_1-\tfrac{1243}{2} d^3 |\beta|^2-48 d^3 |\beta|
\alpha_1-\tfrac{199}{2} d^3 |\beta| \beta_1 \\
& -9 d^3 |\beta| \beta_2+6 d^3 \alpha_1 \beta_1+\tfrac{45}{2} d^3
\beta_1^2-458 d^2 |\beta|^3-\tfrac{57}{2} d^2 |\beta|^2 \alpha_1-9 d^2
|\beta|^2 \alpha_2+116 d^2 |\beta|^2 \beta_1-23 d^2 |\beta|^2 \beta_2
\\
& +3 d^2 |\beta| \alpha_1^2+\tfrac{95}{2} d^2 |\beta| \alpha_1
\beta_1+\tfrac{105}{2} d^2 |\beta| \beta_1^2+d^2 |\beta| \beta_1
\beta_2-d^2 \alpha_1 \beta_1^2-2 d^2 \beta_1^3+\tfrac{155}{2} d
|\beta|^3 \alpha_1-8 d |\beta|^3 \alpha_2 \\
& +\tfrac{155}{2} d |\beta|^3 \beta_1-8 d |\beta|^3 \beta_2+12 d
|\beta|^2 \alpha_1^2+\tfrac{61}{2} d |\beta|^2 \alpha_1 \beta_1+d
|\beta|^2 \alpha_1 \beta_2+d |\beta|^2 \alpha_2 \beta_1+\tfrac{37}{2}
d |\beta|^2 \beta_1^2+2 d |\beta|^2 \beta_1 \beta_2 + \\
 \end{split}
 \end{equation*}
 \begin{equation*}
 \tiny
 \begin{split}
& -\tfrac{3}{2} d |\beta| \alpha_1^2 \beta_1-\tfrac{11}{2} d |\beta|
\alpha_1 \beta_1^2-4 d |\beta| \beta_1^3-\tfrac{5}{2} |\beta|^3
\alpha_1^2+|\beta|^3 \alpha_1 \alpha_2-5 |\beta|^3 \alpha_1
\beta_1+|\beta|^3 \alpha_1 \beta_2+|\beta|^3 \alpha_2
\beta_1-\tfrac{5}{2} |\beta|^3 \beta_1^2 \\
& +|\beta|^3 \beta_1 \beta_2-\tfrac{1}{2} |\beta|^2 \alpha_1^3-3
|\beta|^2 \alpha_1^2 \beta_1-\tfrac{9}{2} |\beta|^2 \alpha_1
\beta_1^2-2 |\beta|^2 \beta_1^3+\tfrac{1243}{3} d^3
|\beta|+\tfrac{70}{3} d^3 \beta_1+6 d^3 \beta_2+1374 d^2 |\beta|^2 \\
& +19 d^2 |\beta| \alpha_1+6 d^2 |\beta| \alpha_2-\tfrac{845}{2} d^2
|\beta| \beta_1+48 d^2 |\beta| \beta_2-27 d^2 \alpha_1 \beta_1-40 d^2
\beta_1^2-2 d^2 \beta_1 \beta_2-\tfrac{842}{3} d |\beta|^3 \\
& -\tfrac{465}{2} d |\beta|^2 \alpha_1+24 d |\beta|^2 \alpha_2-396 d
|\beta|^2 \beta_1+29 d |\beta|^2 \beta_2+d |\beta|^2 \beta_3-8 d
|\beta| \alpha_1^2-33 d |\beta| \alpha_1 \beta_1-3 d |\beta| \alpha_1
\beta_2 \\
& -3 d |\beta| \alpha_2 \beta_1+2 d |\beta| \beta_1^2-11 d |\beta|
\beta_1 \beta_2+d \alpha_1^2 \beta_1+7 d \alpha_1
\beta_1^2+\tfrac{47}{6} d \beta_1^3-\tfrac{92}{3} |\beta|^3 \alpha_1-6
|\beta|^3 \alpha_2+|\beta|^3 \alpha_3 \\
& -\tfrac{92}{3} |\beta|^3 \beta_1-6 |\beta|^3 \beta_2+|\beta|^3
\beta_3+\tfrac{15}{2} |\beta|^2 \alpha_1^2-3 |\beta|^2 \alpha_1
\alpha_2+\tfrac{87}{2} |\beta|^2 \alpha_1 \beta_1-6 |\beta|^2 \alpha_1
\beta_2-6 |\beta|^2 \alpha_2 \beta_1+36 |\beta|^2 \beta_1^2 \\
& -9 |\beta|^2 \beta_1 \beta_2+\tfrac{1}{3} |\beta|
\alpha_1^3+\tfrac{11}{2} |\beta| \alpha_1^2 \beta_1+13 |\beta|
\alpha_1 \beta_1^2+\tfrac{47}{6} |\beta| \beta_1^3-916 d^2 |\beta|+303
d^2 \beta_1-28 d^2 \beta_2+842 d |\beta|^2 \\
& +155 d |\beta| \alpha_1-16 d |\beta| \alpha_2+\tfrac{1237}{2} d
|\beta| \beta_1-31 d |\beta| \beta_2-3 d |\beta| \beta_3+8 d \alpha_1
\beta_1+2 d \alpha_1 \beta_2+2 d \alpha_2 \beta_1-\tfrac{103}{2} d
\beta_1^2 \\
& +14 d \beta_1 \beta_2+706 |\beta|^3+92 |\beta|^2 \alpha_1+18
|\beta|^2 \alpha_2-3 |\beta|^2 \alpha_3-46 |\beta|^2 \beta_1+48
|\beta|^2 \beta_2-6 |\beta|^2 \beta_3-5 |\beta| \alpha_1^2 \\
& +2 |\beta| \alpha_1 \alpha_2-\tfrac{197}{2} |\beta| \alpha_1
\beta_1+11 |\beta| \alpha_1 \beta_2+11 |\beta| \alpha_2
\beta_1-\tfrac{271}{2} |\beta| \beta_1^2+26 |\beta| \beta_1 \beta_2-3
\alpha_1^2 \beta_1-12 \alpha_1 \beta_1^2-10 \beta_1^3 \\
& -\tfrac{1684}{3} d |\beta|-\tfrac{808}{3} d \beta_1+10 d \beta_2+2 d
\beta_3-2118 |\beta|^2-\tfrac{184}{3} |\beta| \alpha_1-12 |\beta|
\alpha_2+2 |\beta| \alpha_3+\tfrac{1184}{3} |\beta| \beta_1-102
|\beta| \beta_2 \\
& +11 |\beta| \beta_3+63 \alpha_1 \beta_1-6 \alpha_1 \beta_2-6 \alpha_2 \beta_1+150 \beta_1^2-24 \beta_1 \beta_2+1412 |\beta|-362 \beta_1+60 \beta_2-6 \beta_3.\\
\end{split}
\end{equation*}

\bibliographystyle{amsplain}
\bibliography{References_Florian}

  \end{document}